\documentclass[11pt, reqno]{amsart}
\usepackage[utf8]{inputenc}
\usepackage[english]{babel}
 
\usepackage[square,sort,comma,numbers]{natbib}

\usepackage[all,cmtip]{xy}
\usepackage{mathalfa}
\usepackage[dvipsnames]{xcolor}
\usepackage{float}
\usepackage{mathrsfs}
\usepackage{graphicx}
\usepackage{tikz}
\usetikzlibrary{arrows}
\usetikzlibrary{matrix,arrows,decorations.pathmorphing}
\usepackage{pinlabel}

\usepackage{amsmath}
\usepackage{amssymb}
\usepackage{amsthm}

\usepackage{stmaryrd}
\usepackage[all]{xy}
\usepackage[mathcal]{eucal}
\usepackage{verbatim}
\usepackage[top=1in,bottom=1in,left=1in,right=1in]{geometry} 
\usepackage{wrapfig}
\usepackage{lipsum}
\usepackage[all]{xy}
\usepackage{color}
\definecolor{myblue}{RGB}{102,153, 255}
\definecolor{myred}{RGB}{204,0,0}
\definecolor{mygreen}{RGB}{0,204,0}
\definecolor{myorange}{RGB}{255,102,0}
\definecolor{mypurple}{RGB}{138,43,226}

\usepackage{hyperref} 
\hypersetup{          
	colorlinks=true, breaklinks, linkcolor=[RGB]{51 102 204}, filecolor=Orchid, urlcolor=[RGB]{51 102 204},
	citecolor=Orchid, linktoc=all, }
\usepackage[nameinlink]{cleveref}

\theoremstyle{plain}
\newtheorem{thm}{Theorem}[section]

\newtheorem{lem}[thm]{Lemma}

\newtheorem{prop}[thm]{Proposition}

\newtheorem{claim}[thm]{Claim}

\newtheorem{conj}[thm]{Conjecture}
\newtheorem{qu}[thm]{Question}
\theoremstyle{definition}

\theoremstyle{remark}

\newcommand{\nc}{\newcommand}
\nc{\dmo}{\DeclareMathOperator}
\DeclareMathOperator{\PConf}{PConf}

\DeclareMathOperator{\Stab}{Stab}
\DeclareMathOperator{\Diff}{Diff}
\DeclareMathOperator{\Mod}{Mod}

\DeclareMathOperator{\Conf}{Conf}

\DeclareMathOperator{\Aff}{Aff}
\DeclareMathOperator{\PB}{PB}

\dmo\Hom{Hom}
\nc{\norm}[1]{\lVert #1 \rVert}
\dmo{\PSL}{PSL}
\dmo\im{Im}
\dmo\Isom{Isom}
\dmo\parenconf{(P)Conf}
\dmo\Jac{Jac}

\newcommand{\R}{\mathbb{R}}
\newcommand{\Z}{\mathbb{Z}}
\newcommand{\C}{\mathbb{C}}
\newcommand{\D}{\mathbb{D}}
\newcommand{\pair}[1]{\langle #1 \rangle}
\nc{\CP}{\mathbb{CP}}
\dmo{\Aut}{Aut}
\nc{\Q}{\mathbb{Q}}
\nc\abs[1]{\left | #1 \right |}

\nc{\para}[1]{\medskip\noindent\textbf{#1.}}

\title[Holomorphic maps between configuration spaces of Riemann surfaces]{\boldmath Holomorphic maps between configuration spaces of Riemann surfaces}
\author{Lei Chen and Nick Salter}
\date{April 14, 2023}
\address{LC: Department of Mathematics, University of Maryland, 4176 Campus Drive, College Park, MD 20742}
\email{chenlei@umd.edu}
\address{NS: Department of Mathematics, University of Notre Dame, Hurley Hall, Notre Dame, IN 46556}
\email{nsalter@nd.edu}

\begin{document}
\maketitle
\vspace{-0.2in}
\begin{abstract}
We prove a suite of results classifying holomorphic maps between configuration spaces of Riemann surfaces; we consider both the ordered and unordered setting as well as the cases of genus zero, one, and at least two. We give a complete classification of all holomorphic maps $\operatorname{Conf}_n(\mathbb{C})\to \operatorname{Conf}_m(\mathbb{C})$ provided that $n\ge 5$ and $m\le 2n$ extending the Tameness Theorem of Lin, which is the case $m = n$. We also give a complete classification of holomorphic maps between ordered configuration spaces of Riemann surfaces of genus at most one (answering a question of Farb), and show that the higher genus setting is closely linked to the still-mysterious ``effective de Franchis problem''. The main technical theme of the paper is that holomorphicity allows one to promote group-theoretic rigidity results to the space level. 
\end{abstract}

\section{Introduction}

Let $\PConf_n(X)$ denote the space of $n$ ordered distinct points on a manifold $X$, and let $\Conf_n(X)$ denote the corresponding space of unordered tuples. If $X$ has a complex structure, then $\PConf_n(X)$ and $\Conf_n(X)$ inherit complex structures from $X$. In this paper we take $X,Y$ to be Riemann surfaces, and consider the family of problems of classifying holomorphic maps $h: \parenconf_n(X) \to \parenconf_m(Y)$. 

As discussed in more detail below, a complete answer in full generality seems to be out of reach, involving unresolved questions related to the ``effective de Franchis problem'' of enumerating and bounding the set of all holomorphic maps $X\to Y$ as well as delicate group-theoretic considerations. However, we are able to address a good portion of the general problem, especially in the ordered setting. 

A general phenomenon recurring throughout our results is that of ``twisting''. Given a holomorphic map $h: \parenconf_n(X) \to \parenconf_m(Y)$, suppose a holomorphic map $A: \parenconf_n(X) \to \Aut(Y)$ is given, where $\Aut(Y)$ is the group of holomorphic automorphisms of $Y$, in our setting itself a complex manifold. Then the {\em twist} $h^A$ of $h$ is defined via the formula 
\[
h^A(x_1, \dots, x_n) = A(x_1, \dots, x_n)(h(x_1, \dots, x_n)).
\] 
Note that the affine twist of a constant map need not be constant. In the case $Y = \C$, the relevant automorphism group is the affine group $\Aff = \{az+b \mid a \in \C^*, b \in \C\}$, in which case we call this an {\em affine twist}. We also note that $\PConf_n(X)$ admits a family of automorphisms given by permuting the coordinates; for ease of stating our main results, we will also consider this a kind of twist. 

The following summarizes our main results; see the indicated statements for precise details. 

\begin{itemize}
\item \textbf{\Cref{main}}: {\em For $m \ge 5$ and $n \le 2m$, up to affine twisting, every holomorphic map $h: \Conf_n(\C) \to \Conf_m(\C)$ is either constant, the identity, or a ``root map'' (see \Cref{section:confc}).}
\item \textbf{\Cref{Ccase}}: {\em For $m \ge 2$, up to a slight generalization of affine twisting, every holomorphic map $h: \PConf_n(\C) \to \PConf_m(\C)$ is either constant or a forgetful map.}
\item \textbf{\Cref{CPcase}}: {\em For $ n\ge 3$, up to twisting, every holomorphic map $h: \PConf_n(\CP^1) \to \PConf_m(\CP^1)$ is either constant or a forgetful map.}
\item \textbf{\Cref{thmtorus}}: {\em Let $X,Y$ be compact Riemann surfaces of genera $g(X) = g(Y) =1$. Then $X\cong Y$ and up to twisting, every holomorphic map $h: \PConf_n(X) \to \PConf_m(Y)$ is a forgetful map.}
\item \textbf{\Cref{thmtorus2}}: {\em Let $X,Y$ be compact Riemann surfaces with $g(X) \ge 2$ and $g(Y) =1$. Then up to twisting, every holomorphic map $h: \PConf_n(X) \to \PConf_m(Y)$ is constant.}
\item \textbf{\Cref{mainpure}}: {\em Let $X,Y$ be compact Riemann surfaces each of genus at least two. Then up to twisting, either $X \cong Y$ and $h$ is a forgetful map, or else $h$ factors as the composition of a forgetful map $p: \PConf_n(X) \to X$ and a holomorphic map $f: X \to \PConf_m(Y)$.}
\end{itemize}

We note that in the case $m = n$, \Cref{main} was previously established by Lin \cite[Theorem 1.4]{Lin}. Farb \cite[Problem 2.4]{farbsurvey} has asked for a classification of holomorphic maps $\PConf_n(\C) \to \PConf_m(\C)$, which is resolved in \Cref{Ccase}. 

Observe that in \Cref{mainpure}, a holomorphic map $f: X \to \PConf_m(Y)$ is the same thing as an $m$-tuple $f_1, \dots, f_m$ of holomorphic maps $f_i: X \to Y$ with the properties that the graphs are pairwise disjoint in $X \times Y$. The study of the set $\operatorname{Hol}(X,Y)$ of nonconstant holomorphic maps $X \to Y$ falls under the purview of the classical {\em de Franchis theorem}, which states that when $g(Y) \ge 2$, the set is finite. Unlike its cousin the Hurwitz theorem (the case $X = Y$), there is a very large gap between known upper bounds on $\abs{\operatorname{Hol}(X,Y)}$ as a function of the genera $g(X), g(Y)$, and the lower bounds achieved by examples: by \cite{chamizo}, there is an upper bound that is slightly super-exponential in $g(X)$, whereas, to the authors' knowledge, there are no families of examples where the number of holomorphic maps grows faster than linearly with $g(X)$. However, we show that the extra constraint of having disjoint graphs is highly restrictive, which may be of independent interest:

\begin{itemize}
\item \textbf{\Cref{mainbound}}: {\em Let $X,Y$ be compact Riemann surfaces each of genus at least two, and let $f: X \to \PConf_m(Y)$ be holomorphic. Then $m \le 4 g(X) g(Y)$.}
\end{itemize}

Our final main result is of a slightly different flavor, considering instead the holomorphic rigidity properties of $\Conf_n(\C)$ into the moduli space of curves $\CMcal M_g$. One such map is given by the hyperelliptic embedding $H: \Conf_n(\C)\to \CMcal{M}_g$ for $g= [\frac{n-1}{2}]$: 
\[
H(\{x_1,...,x_n\}) = \text{the algebraic curve } \{y^2=(x-x_1)...(x-x_n)\}
\]

\begin{itemize}
\item \textbf{\Cref{main2}}: {\em For $n\ge 26$ and $g\le n-2$, if $h:\Conf_n(\C)\to {\CMcal M}_g$ is a non-constant holomorphic map of orbifolds, then $h$ is the hyperelliptic embedding.}
\end{itemize}

\para{The holomorphic landscape} To put the results of this paper into better context, we describe here what is known about the totality of holomorphic maps between configuration spaces. We organize the problem along two axes: first, whether the configurations are ordered or not, and secondly by the genera of the Riemann surfaces $X$ and $Y$. An entry is left blank if we neither know of any interesting examples nor a classification. The ``covering construction'' mentioned below is the following:  let $p: Y \to X$ be an unbranched covering of compact Riemann surfaces of degree $d$. Then for all $n \ge 1$, there is a holomorphic map $P: \Conf_n(X) \to \Conf_{dn}(Y)$, given by $P(\{x_1, \dots, x_n\}) = p^{-1}(\{x_1, \dots, x_n\})$. 

\[
\begin{array}{|c|c|c|c|} \hline
\PConf 	& g(Y) =0 & 1 & \ge 2 \\ \hline
 g(X) = 0 	& \mbox{Fully classified:} & \mbox{All constant:} & \mbox{All constant:}\\
  		& \mbox{Theorems \ref{Ccase},\ref{CPcase}}&\mbox{Proposition \ref{lowtohigh}}&\mbox{Proposition \ref{lowtohigh}} \\ \hline
 1 		& 					& \mbox{Fully classified:}	& \mbox{All constant:}\\
 		&					&\mbox{Theorem \ref{thmtorus}}	&\mbox{Proposition \ref{lowtohigh}}\\ \hline
 \ge 2 	& 					& \mbox{Classified modulo}							&\mbox{Classified modulo}  \\
 		&					& \mbox{case $m=1$:}							&\mbox{understanding }\operatorname{Hol}(X,Y):\\
		&					&\mbox{Theorem \ref{thmtorus2}}							& \mbox{Theorem \ref{mainpure}}\\
  \hline
\end{array}
\]
\[
\begin{array}{|c|c|c|c|} \hline
\Conf 	& g(Y) =0 & 1 & \ge 2 \\ \hline
 g(X) = 0 	& \mbox{Partially classified:} &\mbox{All constant:}&\\
  		& \mbox{Theorem \ref{main}}&\mbox{Proposition \ref{clowhigh}}&\\ \hline
 1 		&					& \mbox{Abundant:}	& \\
 		&					&\mbox{covering construction}	&\\ \hline
 \ge 2 	& 					& 									&\mbox{Abundant:}  \\
 		&					& 									&\mbox{covering construction}\\
  \hline
\end{array}
\]
We think that filling in the remaining entries of these tables is a worthwhile goal of future research. In particular, we would like to highlight the following.
\begin{qu}
For $g(X), g(Y) \ge 2$, does every holomorphic map $h: \Conf_n(X) \to \Conf_m(Y)$ arise via the covering construction?
\end{qu}

\para{Rigidity: holomorphic vs. continuous} While it is hard to find holomorphic maps between configuration spaces of Riemann surfaces, there are many homotopy classes of {\em continuous} maps between them. For instance, there are continuous maps $\Conf_n(\C) \to \Conf_{n+k}(\C)$ by adding $k$ points ``near infinity'', a ``doubling map'' $\Conf_n(\C) \to \Conf_{2n}(\C)$ which replaces a configuration of points with two juxtaposed copies, and many more complicated examples. Even in the case when the induced map $\pi_1(\Conf_n(\C)) \to \pi_1(\Conf_m(\C))$ factors through $\Z$, there are many possibilities, encompassing all of the possible Nielsen-Thurston types of the associated mapping class of the $m$-punctured plane. However, according to \Cref{main}, none of the above can be induced by a holomorphic map, so long as we are in the range $m \le 2n$.  Work of the first author, Kordek, and Margalit \cite{CKM} (cf. \Cref{CKM}) explores this phenomenon and gives a complete classification of homotopy classes of continuous maps $\Conf_n(\C) \to \Conf_m(\C)$ in the range $m \le 2n$; this is the source of the restriction $m \le 2n$ in \Cref{main}. This is accomplished by giving a classification of homomorphisms $B_n \to B_m$, where $B_n = \pi_1(\Conf_n(\C))$ is the braid group on $n$ strands. Following \cite{CKM}, there is now a conjectural classification of {\em all} homomorphisms $\rho: B_n \to B_m$ for $n \ge 5$. To formulate it, we recall that one can also think of the braid group as $B_n \cong \Mod(D_n)$, the mapping class group of the $n$-punctured disk.

\begin{conj}[Reducibility conjecture]\label{redconj}
For $n \ge 5$ and $m > n$, every homomorphism $\rho: B_n \to B_m$ either has cyclic image or else is {\em reducible}, i.e. there is a nonempty set of disjoint essential, non-boundary-parallel curves on the $m$-punctured disk invariant under $\rho(B_n)$. 
\end{conj}

As discussed in \cite{CKM}, the reducibility conjecture in fact implies a stronger classification of homomorphisms, predicting that all homomorphisms are recursively assembled from a small list of basic ones using a certain set of operations. Our methods here can be used to show that \Cref{redconj} implies a {\em complete} classification of all holomorphic maps between configuration spaces.

\begin{thm}\label{strongmain}
If \Cref{redconj} holds, then the statement of \Cref{main} holds for all $n\ge 5$ with no restriction on $m$.
\end{thm}

\para{Prior Results}
Within the world of moduli spaces of complex-analytic objects, there have been various efforts to understand the relationship between their behavior in the holomorphic and smooth categories. The results here expand on the work \cite{Lin} of Lin mentioned above, who obtained \Cref{main} in the case $m = n$, and also investigated the corresponding questions in the setting of configuration spaces on $\CP^1$. Our proof is independent, making use of the developments of \cite{CKM} that give a classification of maps between braid groups, as well as employing some powerful results from Teichm\"uller theory. 

Two other results, while not bearing directly on the results of this paper, were a source of inspiration and merit mention. Antonakoudis--Aramayona--Souto \cite{AAS} give a classification of holomorphic maps $\CMcal{M}_g\to \CMcal{M}_h$ for $h\le  2g-2$, showing that the only non-constant map is the identity.  Letting $\CMcal{A}_h$ be the moduli space of $h$-dimensional, principally-polarized abelian varieties, a recent result of Farb \cite{Farb} shows that a nonconstant holomorphic map $f: \CMcal{M}_g\to \CMcal{A}_h$ exists for $h \le g$ iff $h=g$ and $f$ is the map sending a curve to its Jacobian. It is interesting to note that in both these results, {\em every} smooth map is homotopic to a holomorphic map. This is in contrast to the setting of this paper, where smooth maps between configuration spaces exist in relative abundance (see above). 

\para{Proof Strategy}
The main theme running through our arguments is the promotion of group-theoretic rigidity statements to the level of holomorphic maps. In \cite{ChenPure}, the first author has established a general rigidity theorem classifying homomorphisms $f: PB_n(X) \to \Lambda$, where $X$ is a Riemann surface, $PB_n(X) := \pi_1(\PConf_n(X))$ denotes the space of ordered configurations of $n$ points, and $\Lambda$ is a torsion-free nonelementary hyperbolic group. When $f$ is induced from a holomorphic map $F$ between complex manifolds, principles of complex analysis can be used to show that $F$ must have a specific form.  In other settings, we make use of rigidity phenomena for holomorphic maps of Riemann surfaces into the moduli space of Riemann surfaces, equivalently the rigidity of holomorphically-varying families of Riemann surfaces over Riemann surfaces. We will use two such theorems in our proof of \Cref{main}, both concerning the {\em monodromy} of such families, the homomorphism $\rho: \pi_1(B) \to \Mod(S)$ from the fundamental group of the base to the mapping class group of the fiber. Specifically, we will use a result of Daskalopoulos-Wentworth \cite{DW} showing that the monodromy of a non-isotrivial family of curves is necessarily ``rich'' in a certain technical sense, and a result of Imayoshi--Shiga \cite{IS} stating that if two families have the same monodromy, then they are equal to each other. Inside $\Conf_n(\C)$, there are many embedded Riemann surfaces: given distinct points $Y\in \Conf_{n-1}(\C)$, there is an associated embedding $i_Y: \C-Y\to \Conf_n(\C)$ of the finite-type surface $\C-Y$ into $\Conf_n(\C)$. The holomorphic map $\Conf_n(\C) \to \Conf_m(\C)$ thus equips each such surface with a family of Riemann surfaces (specifically, the fiber is $\C$ punctured at $m$ points). The monodromy of such families factors through the homomorphism $B_n \to B_m$ induced by the map on the configuration spaces; we exploit the classification of \cite{CKM} along with the criterion of Daskalopoulos-Wentworth to rule out many possibilities, and the rigidity result of Imayoshi--Shiga gives control over the situations where a map is possible. The phenomenon of affine twisting appears because the target space $\Conf_m(\C)$ is not exactly the moduli space of $m$ points in $\C$, but becomes a finite cover of this after erasing the action of the affine group $\Aff$.

\para{Acknowledgements}
NS is supported by NSF Award No. DMS-2153879. LC is supported by NSF Award No. 2203178 and Sloan Foundation. We thank Peter Huxford for pointing a problem in the proof of Theorem 3.1.

\section{background}
\subsection{Classification of homomorphisms between braid groups}
The basic strategy of the proof of the main results is to first understand all possible induced maps on the fundamental groups (i.e. braid groups), and then to contrast this with known rigidity results about holomorphic maps between the associated moduli spaces. In this section we carry out the first of these tasks.

\Cref{CKM} below is a corollary of the work of Chen--Kordek--Margalit \cite{CKM}. To state it, we introduce the following terminology. A {\em transvection} of a group homomorphism $f: G \to H$ is a homomorphism $f^t: G \to H$ defined via
\[
f^t(g) = f(g) t^{\ell(g)},
\]
where $\ell: G \to \Z$ is a homomorphism and $t \in H$ centralizes $f(G)$. A homomorphism $f: G \to B_n$ is said to be {\em reducible} if the image $f(G)$ preserves a finite set of isotopy classes of essential non-boundary-parallel curves on the $n$-punctured disk (viewing $B_n$ as the mapping class group of the punctured disk).  Lastly, an element of $B_n$ is said to have {\em prefinite order} if its image in $B_n/Z_n$ has finite order.

\begin{thm}\label{CKM}
For $n\ge 5$ and $m \le 2n$, let $\rho:B_n\to B_{m}$ be a homomorphism. Then exactly one of the following conditions hold:
\begin{enumerate}
\item
$n = m$ and $\rho$ is the identity map up to transvection and a pre-composition of an automorphism of $B_n$, 
\item 
$\rho$ is either reducible and nontrivial, or else has infinite cyclic image generated by a pseudo-Anosov,  
\item
$\rho$ has prefinite cyclic image.
\end{enumerate}
\end{thm}

\begin{proof}
Note first that the statement is equivalent to the assertion that at least one of the following conditions hold:
\begingroup
\renewcommand\labelenumi{(\theenumi')}
\begin{enumerate}
\item
$n = m$ and $\rho$ is the identity map up to transvection,
\item 
$\rho$ is reducible (possibly trivial)
\item
$\rho$ has cyclic image.
\end{enumerate}
\endgroup
The latter assertion follows more readily from \cite[Theorem 1.1]{CKM}, but we will see in the proof of \Cref{main} that the former organization corresponds to the classification of holomorphic maps.

\cite[Theorem 1.1]{CKM} asserts that for $n \ge 5$, any homomorphism $\rho: B_n \to B_{2n}$ is a transvection of one of five ``standard homomorphisms'', and \cite[Corollary 1.2]{CKM} asserts that every $\rho: B_n \to B_m$ for $m < 2n$ is one of two of the five types possible when $m = 2n$. We first argue that if $\rho$ satisfies at least one of the conditions, then so does any transvection of $\rho$; then we will see that each standard homomorphism satisfies one of the conditions. 

There is nothing to check if $\rho$ satisfies condition (1'). Suppose now that $\rho$ is reducible. If $\rho$ is the trivial homomorphism, then any transvection of $\rho$ has cyclic image (note that if $t$ is pseudo-Anosov, the transvection will not necessarily be reducible). If $\rho$ is reducible and nontrivial, then $\rho$ has a nonempty ``canonical reduction system''  $\{C_i\}$ of disjoint essential, non-boundary-parallel curves (see e.g. \cite{BLM}). Given a transvection $\rho^t$, the element $t$ commutes with every element of $\rho$, from which it follows that $t$ preserves $\{C_i\}$ as well, so that $\rho^t$ is likewise reducible. If $\rho$ has cyclic image, then $\rho$ factors through the abelianization map $\ell: B_n \to \Z$, and hence by construction any transvection of $\rho$ does as well. 

We now see that each of the five ``standard homomorphisms'' described in \cite{CKM} satisfies one of the above conditions. Each of these is a routine verification; for brevity's sake we will assume familiarity with the five, as given in \cite[page 1]{CKM}. The trivial homomorphism satisfies conditions (2') and (3'), the inclusion homomorphism satisfies (1') if $m = n$ and (2') if $m > n$, and the diagonal, flip-diagonal, and $k$-twist cabling maps are all reducible (condition (2')). 
\end{proof}

In anticipation of \Cref{main2}, we next consider the case of homomorphisms $\rho: B_n \to \Mod(S_g)$ (here $\Mod(S_g)$ denotes the mapping class group of a closed surface of genus $g$). Let $H$ be the hyperelliptic embedding defined in the introduction; $H_*$ is then the induced homomorphism on orbifold fundamental groups. We have the following result of Chen--Mukerjea \cite[Theorem 1.1]{Aru}.

\begin{thm}\label{BnMod}
For $n\ge 26$ and $g\le n-2$, let $h:B_n\to \Mod(S_g)$ be 
a homomorphism. Then up to transvection and a pre-composition of an automorphism of $B_n$, $h$ is either trivial or $H_*$.
\end{thm}

\subsection{Some results from complex analysis and Teichm\"uller theory}
Before considering the holomorphic theory of {\em families} of Riemann surfaces, we first mention a useful variant of the removable singularity theorem that we will employ throughout the paper.

\begin{prop}\label{extend}
Let $X, Y$ be Riemann surfaces of finite type, each of negative Euler characteristic, and let $f: X \to Y$ be holomorphic. Then $f$ admits a holomorphic extension $F: \bar X \to \bar Y$, where $\bar X, \bar Y$ denote the compact Riemann surfaces associated to $X,Y$. 
\end{prop}

\begin{proof}
Since each of $X, Y$ have negative Euler characteristic, they are uniformized by $\D$ and hence admit complete hyperbolic metrics. By the Schwarz-Pick theorem, $f$ is distance non-increasing in these metrics. Let $x \in \bar X - X$ be given, and let $\gamma \subset X$ be a loop encircling $x$. The homotopy class of $\gamma$ admits representatives of arbitrarily short length, and hence the same is true for $f(\gamma)$, showing that $f(\gamma)$ is either null-homotopic or else encircles a puncture in $Y$. Thus $f$ admits a continuous extension $F: \bar X \to \bar Y$, and by the usual removable singularity theorem, it follows that $F$ is holomorphic. 
\end{proof}

We now turn our attention to families of Riemann surfaces. Let $\CMcal{M}_{g,n}'$ be a finite orbifold cover of $\CMcal{M}_{g,n}$. A homomorphism $f: G \to \Mod(S_{g,n})$ is said to be {\em sufficiently large} if the image contains two pseudo-Anosov elements with distinct fixed point sets in $\text{PMF}(S_{g,n})$. We will not need to know the precise meaning of these terms, only that a sufficiently large subgroup is not {\em reducible}, i.e. there is no globally-invariant finite set of curves. Daskalopoulos-Wentworth proved the following \cite[Theorem 5.7]{DW}.

\begin{thm}[Daskalopoulos-Wentworth]
\label{DMresult}
Let $B$ be a Riemann surface of finite type and $f: B\to \CMcal{M}_{g,n}$ be a non-constant holomorphic map. Then $f_* : \pi_1(B) \to \Mod(S_{g,n})$ is sufficiently large.
\end{thm}
Imayoshi--Shiga proved the following \cite[Section 3]{IS}, which can be strengthened as the following. 

\begin{thm}[Imayoshi--Shiga]\label{rigidity}
Let $B$ be a Riemann surface of finite type. If $f,h:B\to \CMcal{M}_{g,n}'$ are non-constant holomorphic maps or anti-holomorphic maps and the monodromy maps $f_*=h_*: \pi_1(B) \to \Mod(S_{g,n})$ coincide, then $f=h$.
\end{thm}
\begin{proof}
\cite[Corollary 5.6]{DW} asserts that if $B$ is a compact domain, then there is a unique harmonic map $B \to \CMcal{M}_{g,n}'$ in the homotopy class of a map with sufficiently large monodromy. This can be extended to $B$ a Riemann surface of finite type by taking an exhaustion $B = \cup B_n$ of nested compact subsurfaces. The claim now follows: holomorphic maps and anti-holomorphic maps are harmonic maps, and by Theorem \ref{DMresult}, the monodromy of holomorphic or anti-holomorphic maps are sufficiently large. \end{proof}

\subsection{Relations between $\Conf_m(\C)$, $\PConf_m(\C)$ and $\CMcal{M}_{0,m,1}$}
Let $\CMcal{M}_{0,m,1}$ denote the moduli space of $m+1$ points on $\CP^1$, where one of the points is distinguished. We now discuss the natural projection map $\pi_m: \Conf_m(\C)\to \CMcal{M}_{0,m,1}$. Let $\Aff$ be the affine group of $\C$, which induces an action on $\Conf_m(\C)$ and on $\PConf_m(\C)$.

\begin{lem}\label{affaction}
The preimage $\pi_m^{-1}(\pi_m(X))$ of a point $X\in \Conf_m(\C)$ is the orbit $\Aff(X)\subset \Conf_m(\C)$. 
\end{lem}

\begin{proof}
The point $\pi_m(X)$ represents $\CP^1$ with distinct marked points $(x_0,\{x_1,...,x_m\})$, the first of which is distinguished. We first apply a Mobi\"{u}s transformation to send the first point into $\infty$. Now $\pi_m(X)$  is represented as $(\infty,\{x_1,...,x_m\})$ where we have $m$ distinct points $x_1,...,x_m\in \C$. Two points $(\infty,\{x_1,...,x_m\})$ and $(\infty,\{y_1,...,y_m\})$ represent the same points in $\CMcal{M}_{0,m,1}$ if they are holomorphically related, i.e. if and only if there is a biholomorphism $f: \C \to \C$ taking $\{x_1 \dots, x_m\}$ to $\{y_1, \dots, y_m\}$. The group of holomorphic automorphisms of $\C$ is $\Aff$, and so $(\infty,\{x_1,...,x_m\})$ and $(\infty,\{y_1,...,y_m\})$ represent the same points in $\CMcal{M}_{0,m,1}$ if and only if there is $f\in \Aff$ such that $f(\{x_1,...,x_m\})=\{y_1,...,y_m\}$.
\end{proof}

Denote by $P\CMcal{M}_{0,m+1}$ the moduli space of $m+1$ {\em ordered} distinct points in $\CP^1$. Similar to \Cref{affaction}, we have the following (cf. \cite[p. 247 ff.]{FM} for the assertion about orbifold fundamental groups). 

\begin{lem}\label{cover}
There is an isomorphism of complex orbifolds
\[
\PConf_m(\C)/\Aff\cong P\CMcal{M}_{0,m+1}.
\] 
The quotient $\PConf_m(\C) \to \PConf_m(\C)/\Aff$ induces the natural map $PB_n \to PB_n/Z_n$ on orbifold fundamental groups. 
\end{lem}

A {\em rotation group} in $\Aff$ centered at $c$ of order $p$ is the cyclic subgroup generated by the element
\[
g(x)=e^{2\pi/p}(x-c)+c.
\]

\begin{lem}\label{rotation}
Any finite subgroup of $\Aff$ is a rotation group centered at some $c\in \C$.
\end{lem}

\begin{proof}
Let $G<\Aff$ be a finite subgroup. Any orientation-preserving finite group action on $\R^2$ is cyclic by \cite{Ker2}, which implies that $G$ is generated by a single element $f(x)=ax+b$ of order $p$. This implies that the linear coefficient $a$ must be a $p^{th}$ root of unity, and then the center can be determined by the formula $c = b/(1-a)$.
\end{proof}

\section{Holomorphic maps between configuration spaces on $\C$: \Cref{main}}\label{section:confc}

Our first main result, \Cref{main}, gives a classification of holomorphic maps $h: \Conf_n(\C) \to \Conf_m(\C)$ in the range $m \le 2n$. To state the result, we must first discuss one of the possible archetypes.

\para{Root maps} There are two maps
\[
r_p: \Conf_k(\C^*)\to \Conf_{kp}(\C^*)
\] 
and 
\[
r_p': \Conf_k(\C^*)\to \Conf_{kp+1}(\C),
\] 
where the first takes $p^{th}$ roots of the $k$ distinct nonzero points, and the second takes the union of  the $p^{th}$ roots of the $k$ distinct nonzero points and $\{0\}$. Such maps are called {\em basic root maps}. 

A map $\Conf_n(\C) \to \Conf_{kp+\epsilon}(\C)$ (with $\epsilon \in \{0,1\}$) is called a {\em root map} if it admits a factorization
\[
\Conf_n(\C) \to \Conf_k(\C^*) \to \Conf_{kp+\epsilon}(\C),
\] 
where the map $\Conf_n(\C) \to \Conf_k(\C^*)$ is a twist of a {\em constant} map by some holomorphic map $A: \Conf_n(\C) \to \C^*$ and the latter map is a basic root map as above. By convention, we consider the zero map $\Conf_k(\C) \to \Conf_1(\C)\cong \C$ to be a root map of the second kind with $p = 0$. 

The main result of this section is the following rigidity result about holomorphic maps. We note that in the case $m = n$, \Cref{main} was established by Lin \cite[Theorem 1.4]{Lin}.

\begin{thm}\label{main}
For $n\ge 5$ and $m\le 2n$, if $h:\Conf_n(\C)\to \Conf_m(\C)$ is a non-constant holomorphic map, then $h$ is either an affine twist of the identity map or a root map. \end{thm}

We will also consider a variant of this result, where the target is the moduli space ${\CMcal M}_g$. One such holomorphic map is given by the hyperelliptic embedding $H: \Conf_n(\C)\to \CMcal{M}_g$ for $g= [\frac{n-1}{2}]$:
\[
H(\{x_1,...,x_n\}) = \text{the algebraic curve } \{y^2=(x-x_1)...(x-x_n)\}
\]

\begin{thm}\label{main2}
For $n\ge 26$ and $g\le n-2$, if $h:\Conf_n(\C)\to {\CMcal M}_g$ is a non-constant holomorphic map, then $h$ is the hyperelliptic embedding. \end{thm}

Note that  \Cref{main2} does not mention affine twists, in contrast to \Cref{main}. The reason is that affine twisting gives equivalent points in the moduli space of the punctured sphere; the hyperelliptic embedding does not depend on the actual location of the points but only on the holomorphic structure of the punctured sphere. Taking the orbifold structure on $\CMcal{M}_g$ into account, it is possible to distinguish between two maps, one being a sort of twist of the other by the hyperelliptic involution; see \Cref{hyp} for details. \\

To prove \Cref{main}, we divide into cases according to \Cref{CKM}. Let 
\[
h:\Conf_n(\C)\to \Conf_n(\C)
\]
be a holomorphic map, and let $h_*$ be the map induced by $h$ on fundamental group. \Cref{CKM} asserts that there are three possibilities for $h_*$; we consider each in turn. 

\subsection{$h_*$ is the identity map up to transvection}\label{case1}

\begin{proof}
By \cite{dg}, the outer automorphism group $\text{Out}(B_n)\cong \Z/2$. The generator is given by the map $\sigma_i \mapsto \sigma_i^{-1}$ that sends each standard generator to its inverse. On the level of $\Conf_n(\C)$, this is induced by the complex conjugation map $\{z_i\} \mapsto \{\bar z_i\}$.  It follows that the composition of $h$ with the forgetful map $\overline{h}: \Conf_n(\C)\xrightarrow{h} \Conf_n(\C)\xrightarrow{F} \CMcal{M}_{0,n+1}$ is homotopic to either the identity map or the complex conjugation map. 

The Fadell-Neuwirth fibration $\PConf_n(\C) \to \PConf_n(\C)$ realizes $\PConf_n(\C)$ as a union of finite-type Riemann surfaces (given as $\C$ with $n-1$ points deleted). Pulling back to $\PConf_n(\C)$ and restricting $\overline{h}$ to each of these fibers, we conclude from the Imayoshi-Shiga theorem \Cref{rigidity} that $\overline{h}$ is precisely equal to either the forgetful map $\Conf_n(\C)\xrightarrow{F} \CMcal{M}_{0,n+1}$ or the precomposition of $F$ by complex conjugation. The second possibility is ruled out because the complex conjugation map is not holomorphic.

We then consider the lift
\[
\tilde{h}:\PConf_n(\C)\to \PConf_n(\C).
\]
 Let $\pi_n': \PConf_n(\C)\to \PConf_n(\C)/\Aff$ be the natural quotient map. By \Cref{cover}, the space $\PConf_n(\C)/\Aff$ is a finite orbifold cover of $\CMcal{M}_{0,n+1}$ and $\pi_1^{orb}(\PConf_n(\C)/\Aff) \cong PB_n / Z_n$, with $\pi_{n,*}'$ the natural quotient $PB_n \to PB_n/Z_n$. 
 
Let $Y = (x_1,...,x_{n-1})$ be an $n-1$-tuple of ordered distinct points in $\C$. There is an embedding 
\[
i_Y: \C-Y\to \PConf_n(\C)
\]
such that $i_Y(x)=(Y,x)$. From the above description of $\pi_{n,*}'$, we see that both $\pi_n'\circ \tilde{h}\circ i_Y$  and $\pi_n'\circ i_Y$ induce the same map on the orbifold fundamental group. By \Cref{rigidity}, it follows that 
\[
\pi_n'\circ \tilde{h}\circ i_Y = \pi_n'\circ i_Y.
\]
Therefore $\tilde{h}(Y,x) = g(Y,x)(Y,x)$ for some function $g: \PConf_n(\C) \to \Aff$. 

\begin{claim}
The map $g:\PConf_n(\C)\to \Aff$ is holomorphic and factors through $\Conf_n(\C)$.
\end{claim}
\begin{proof} For $X = (x_1, \dots, x_n) \in \PConf_n(\C)$, we write
\[
\tilde h(X) = (\tilde h_1 (X), \dots, \tilde h_n(X))
\]
with
\[
\tilde h_i (X) = g(X) (x_i).
\]
Writing $g(X)(z) = az + b$, we can solve for $a, b$ using the fact that for {\em any} pair of distinct indices $i,j$, the map $g(X)$ takes $x_i$ to $\tilde h_i(X)$ and $x_j$ to $\tilde h_j(X)$. This yields the expressions 
\[
a = \frac{\tilde h_i(X) - \tilde h_j(X)}{x_i-x_j}, \qquad b = \frac{x_j\tilde h_i(X) - x_i \tilde h_j(X) }{x_i-x_j}.
\]
Observe that since $\tilde h$ is holomorphic, these expressions vary holomorphically with $X$, and that they are independent of $i$ and $j$ by the assumption that $\tilde h_i(X) = g(X)(x_i)$ for all $i$.
\end{proof}

Thus in this case, we have shown that the original map $h$ is the affine twist of the identity map by the holomorphic map $g: \Conf_n(\C) \to \Aff$.

\subsection{The image of $h_*$ is reducible or else infinite cyclic pseudo-Anosov}\label{case2}
 Given a set $Y = \{x_1,...,x_{n-1}\}$ of $n-1$ distinct points in $\C$, there is an embedding 
\[
i_Y: \C-Y\to \Conf_n(\C)
\]
such that $i_Y(x)=Y\cup \{x\}$.
Composing $i_Y$ with the natural projection $\pi_m: \Conf_m(\C)\to \CMcal{M}_{0,m,1}$, we obtain a holomorphic map $h_Y:= \pi_m\circ h\circ i_Y$ from the finite-type Riemann surface $\C-Y$ into $\CMcal{M}_{0,m,1}$.

Suppose first that $h_*$ has cyclic image generated by a pseudo-Anosov. As the kernel of $\pi_{m,*}: \pi_1(\Conf_m(\C)) \to \pi_1^{orb}(\CMcal{M}_{0,m,1})$ is contained in the center $Z_n \le B_n = \pi_1(\Conf_m(\C))$ by \cite[p. 247 ff.]{FM}, it follows that $h_{Y,*}$ also has cyclic image generated by a pseudo-Anosov element. By \Cref{DMresult}, this is not possible. Likewise, in the case when $h_*$ has reducible image, the induced map $h_{Y*}$ also has nontrivial reducible image, which is again prohibited by \Cref{DMresult}. 
 
\subsection{$h_*$  has prefinite cyclic image}\label{case3}
We will show that in this case $h$ is a root map up to an affine twist.
\begin{claim}\label{constantclaim}
The map $\pi_m\circ h$ is a constant map.
\end{claim}
\begin{proof}
We continue to consider 
\[
h_Y: \C-Y \to \CMcal{M}_{0,m,1}
\]
as in \Cref{case2}. If $h_Y$ were not constant, then it would determine a locally nontrivial family, which would then have sufficiently large monodromy by \Cref{DMresult}, contrary to hypothesis. Therefore $\pi_m\circ h(\{x_1,...,x_n\})=h_{\{x_1,...,x_{n-1}\}}(x_{n})$, which does not depend on $x_n$. This implies that it does not depend on any coordinate, by symmetry. The claim follows.
\end{proof}

Denote the image of $\pi_m\circ h$ by $X\in \CMcal{M}_{0,m,1}$, and let $X_0=\{x_1,...,x_m\}$ be a representative of $X$ in $\Conf_m(\C)$. By \Cref{cover}, the image $\pi_m^{-1}(X)$ is given as the orbit $\Aff(X_0)$, and by the orbit-stabilizer theorem, there is an isomorphism of complex manifolds, defining $G:= \Stab(X_0)$, 
\[
\Aff(X_0) \cong \Aff/G.
\]
Thus $h$ is given as a holomorphic map $h: \Conf_n(\C) \to \Aff/G$. By \Cref{rotation}, $G$ is a rotation group of order $p$ with center $c$ and $X_0$ is a $G$-invariant subset of $\C$. By applying an affine twist, we can assume that the center point of $G$ is $c = 0$, so that $G = \mu_p$, the $p^{th}$ roots of unity, and $X_0$ consists of all $p^{th}$ roots of some fixed subset $Y_0 = \{y_1,...,y_k\} \subset \C^*$, possibly along with $0$. 

There is an isomorphism of complex manifolds
\[
\Aff /G \cong \C^* \times \C
\]
which sends the coset of $z \mapsto az+b$ to $(a^p, b)$. Under this identification, let
\[
A: \Conf_n(\C) \to \C^*
\]
be given as the first coordinate of the map $h: \Conf_n(\C) \to \Aff/G$, and likewise define $B: \Conf_n(\C) \to \C$ as the second coordinate. The map of configuration spaces
\[
h: \Conf_n(\C) \to \Aff(X_0) \subset \Conf_{kp+\epsilon}(\C)
\]
 is now seen to be an affine twist of a root map, factoring as shown below:
 \[
 \xymatrix{
 \Conf_n(\C) \ar[rr]^<>(.5){(c_{Y_0})^A}&&
 \Conf_k(\C^*) \ar[rr]^<>(.5){r_p}&&
 \Conf_{kp+\epsilon} \ar[rr]^<>(.5){id^B}&&
 \Conf_{kp+\epsilon}
 }.
 \]
\end{proof}

\subsection{Proof of \Cref{strongmain}}
\Cref{redconj} asserts that every homomorphism $\rho: B_n \to B_m$ with $5 \le n < m$ is reducible, or else has cyclic image. The arguments of \Cref{case2} and \Cref{case3} then apply to extend the classification to this setting.

\subsection{Proof of \Cref{main2}}\label{hyp}
Suppose $h: \Conf_n(\C) \to \CMcal{M}_g$ is holomorphic, with $n,g$ satisfying the bounds of \Cref{BnMod}. We execute the same strategy as before, now following the cases delineated by \Cref{BnMod}. If $h_*: B_n \to \Mod(S_g)$ is trivial, then we imitate the proof of \Cref{constantclaim} to see that $h$ is constant. If the image of $h_*$ is cyclic, it is not sufficiently large, and so cannot arise from a holomorphic map by \Cref{DMresult}. 

It remains to consider the case where $h_*$ is a transvection of $H_*$ (the map induced by the hyperelliptic embedding $H: \Conf_n(\C) \to \CMcal{M}_g$). Note that the image of $H_*$ is the hyperelliptic mapping class group, which has centralizer $\Z/2\Z$, generated by the hyperelliptic involution $\iota$. Thus, there is {\em exactly one} nontrivial transvection of $H_*$, by $\iota$. If $h_*: B_n \to \Mod(S_g)$ is given by $H_*$ , we follow the argument of \Cref{case1} to see that on each submanifold of the form $\C-Y \subset \Conf_n(\C)$, the maps $h$ and $H$ coincide, hence coincide globally. If $h_* = H_*^\iota$, the same argument shows that there is {\em at most one} holomorphic map in the homotopy class of $h$; it remains to give a construction. As a map of {\em sets}, the underlying map $H^\iota$ is the same as $H$, sending the configuration $\{x_1, \dots, x_n\}$ to the hyperelliptic curve $y^2 = (x-x_1)\dots (x-x_n)$. As a map of {\em complex orbifolds}, the two differ in how the marking is specified on the universal covers: $H^\iota$ arises from $H$ by precomposition with an affine twist by some holomorphic map $\Delta: \Conf_n(\C) \to \C^*$ that induces the abelianization $\Delta_*: B_n \to \Z$.

\section{Rigidity of holomorphic maps between $\PConf_n(\C)$ and $\PConf_n(\CP^1)$}
In this section, we will give a new proof of \cite[Theorem 2.5]{Lin} (classifying holomorphic maps $F: \PConf_n(\C) \to \C-\{0,1\}$) and use this to give a complete classification of holomorphic maps between $\PConf_n(\C)$.

To state the results, we first define some basic ingredients: the maps $sr, cr, RQ,$ and $NI$. The maps $sr$ and $cr$ are holomorphic maps $\PConf_n(\C)\to \C-\{0,1\}$, both arising from the cross-ratio. The ``simple ratio'' $sr(i,j,k):\PConf_n(\C)\to \C-\{0,1\}$ is given by
\[
sr(i,j,k)(x_1,...,x_n)=\frac{x_k-x_i}{x_j-x_i}.
\]
The second map is the cross-ratio $cr(i,j,k,l): \PConf_n(\C)\to \C-\{0,1\}$, given by
\[
cr(i,j,k,l)(x_1,...,x_n)= \frac{x_l-x_i}{x_l-x_k} \frac{x_j-x_k}{x_j-x_i}.
\]

Another interpretation of the maps $sr(i,j,k)$ and $cr(i,j,k,l)$ are the following. For $(x_1,...,x_n)\in\PConf_n(\C)$, there is a unique element in $A\in \Aff$ such that $A(x_i)=0,A(x_j)=1$. We define the value $sr(i,j,k)(x_1,...,x_n) = A(x_k)$. Likewise, for $(x_1,...x_n)\in\PConf_n(\C)$, there is a unique element in $A\in \PSL(2,\C)$ such that \[
A(x_i)=0,A(x_j)=1, A(x_k)=\infty.\] We define the value $cr(i,j,k,l)(x_1,...,x_n) = A(x_l)$. 

For later use, we record the following properties of $cr$ and $sr$ under permutation of indices; these can be checked by direct inspection.

\begin{lem}\label{lem:permute}
There are identities
\[ sr(j,i,k) = 1 - sr(i,j,k), \qquad sr(j,k,i) = 1/sr(i,j,k),
\] and
\[ cr(i,j,k,l) = cr(k,l,i,j), \qquad cr(j,i,k,l) = 1- cr(i,j,k,l), \qquad cr(k,j,i,l) = 1/cr(i,j,k,l).
\]
\end{lem}

We will give a new proof of the following result of Lin \cite[Theorem 2.15]{Lin}, which has the virtue of being somewhat shorter than the original. 
\begin{thm}\label{rigiditytofree}
For $n \ge 3$, any non-constant holomorphic map $f:\PConf_n(\C)\to \C-\{0,1\}$ is given by either $sr(i,j,k)$ or $cr(i,j,k,l)$.
\end{thm}

Using this, we will classify holomorphic maps $h: \PConf_n(\C) \to \PConf_m(\C)$. We define two basic ingredients. The first is the map $RQ_n:\PConf_n(\C)\to \PConf_{n-1}(\C)$, given by the following. For $(x_1,...,x_n)\in \PConf_n(\C)$, note that $cr(x_1, x_2, x_3, z)$, viewed as a M\"obius transformation, sends $x_1, x_2, x_3$ to $0,1, \infty$, respectively. Then define 
\[
RQ_n(x_1,...,x_n) = (0,1,cr(x_1, x_2,x_3, x_4),...,cr(x_1,x_2,x_3,x_n)).
\]
The terminology comes from the fact that the classical ``resolving the quartic'' map $\PConf_4(\C) \to \PConf_3(\C)$ is affine equivalent to $RQ_4$. 

The second basic map we consider is the ``normalized inversion'' $NI_n: \PConf_n(\C) \to \PConf_n(\C)$, given by the formula
\[
NI_n(x_1, \dots, x_n) = \left(0, \frac{1}{x_2-x_1}, \dots, \frac{1}{x_n-x_1} \right).
\]

We can now state the main results of the section.

\begin{thm}\label{Ccase}
Let $m\ge 2$. Let $h: \PConf_n(\C)\to \PConf_m(\C)$ be a holomorphic map. Then up to permutation of coordinates in the source and target and affine twisting, $h$ is a composition of one or more of the following:
\begin{itemize}
\item $h = c$ a constant map,
\item $h = RQ_n$,
\item $h = NI_n$,
\item $h = \pi_n^m$ the forgetful map $(x_1,\dots, x_n) \mapsto (x_1, \dots, x_m)$ (including $m = n$, the identity). 
\end{itemize}
\end{thm}
In particular, there is no holomorphic map $\PConf_n(\C)\to \PConf_m(\C)$ for $1<n<m$.

\begin{thm}\label{CPcase}
Let $m\ge 3$. Let $h: \PConf_n(\CP^1)\to \PConf_m(\CP^1)$ be a holomorphic map. Up to a permutation on the domain and twisting by a holomorphic map $A: \PConf_n(\CP^1) \to \PSL(2,\C)$, either $F$ is a constant map or a forgetful map.
\end{thm}

\subsection{\boldmath Rigidity of $\PConf_n(\C)\to \C-\{0,1\}$ and $\PConf_n(\CP^1)\to \C-\{0,1\}$}

We begin with the following lemma. 
\begin{lem}\label{affinetoD}
Any holomorphic map $f: \Aff\to \C-\{0,1\}$ is constant. Likewise, any holomorphic map $f: \PSL(2,\C)\to \C-\{0,1\}$ is a constant map.
\end{lem}
\begin{proof}
Given $f: \Aff \to \C-\{0,1\}$ holomorphic, there is an induced holomorphic map $F$ between the universal covers. As a complex manifold, $\Aff$ is isomorphic to $\C^* \times \C$ and hence its universal cover is $\C^2$, while the universal cover of $\C-\{0,1\}$ is $\D$; by Liouville's theorem, it follows that $F$, and hence $f$, must be constant.

Now suppose holomorphic $f: \PSL(2, \C) \to \C-\{0,1\}$ is given. The action by M\"obius transformation gives a fiber bundle
\[
\Aff\to \PSL(2,\C)\to \CP^1
\]
whose fibers are holomorphic submanifolds. By the previous paragraph, the restriction of $f$ to each fiber must be constant, and hence
$f$ factors through the base space $\CP^1$; the result now follows via the maximum principle. 
\end{proof}

We will also make use of the following result of \cite[Theorem 1.4]{ChenPure}:

\begin{thm}\label{factor34}
For $n \ge 3$ and $m \ge 2$, any homomorphism
\[
f: PB_n \to F_m
\]
factors through a forgetful map $p_*: PB_n \to PB_m$ with $m \in \{3,4\}$; in the case $m = 4$, there is a further factorization through $RQ_{4,*}: PB_4 \to PB_3$.
\end{thm}

Our other main tool will be the following lemma:

\begin{lem}\label{forgetrigid}
For $n \ge 3$, let
\[
f: \PConf_n(\C) \to \C-\{0,1\}
\]
be holomorphic. If $f_*: PB_n \to F_2$ factors through a forgetful map $p_*: PB_n \to PB_m$, then $f$ factors through a forgetful map $p: \PConf_n(\C) \to \PConf_m(\C)$.
\end{lem}
\begin{proof}
Suppose that $p_*$ forgets the $i^{th}$ strand.  Fixing a configuration of $n-1$ points
 \[
 X_i :=x_1, \dots, \widehat{x_i}, \dots, x_n,
 \]
 the restriction of $f$ to the subspace $\C-X_i \subset \PConf_n(\C)$ (where only the $i^{th}$ coordinate varies) then lifts to the universal cover $\D$ of $\C-\{0,1\}$. However by the removability singularity theorem, this extends to give a holomorphic map $F: \C \to \D$, which must be constant by Liouville's theorem, as desired.
\end{proof}

\begin{proof}[Proof of Theorem \ref{rigiditytofree}]

We proceed by induction on $n$.

\para{\boldmath Base case 1: $n=3$}
For $ \tau \in \C-\{0,1\}$ fixed, define $h_\tau: \Aff \to \C-\{0,1\}$ by
\[
h_\tau(az+b) = h(b, a+b, a\tau + b).
\]
By \Cref{affinetoD}, each $h_\tau $ is constant, with value $c(\tau) \in \C-\{0,1\}$. Since $h$ is holomorphic, so too is the induced map $c: \C-\{0,1\} \to \C - \{0,1\}$.

By the Great Picard Theorem,  the points $0,1,\infty$ are at worst poles of $c$. Thus, $c$ extends to a holomorphic map $\hat c: \CP^1\to \CP^1$, a rational function so that $\hat c^{-1}(\{0,1, \infty\})\subset \{0,1,\infty\}$. If this containment is strict then $\hat c$ has degree zero and hence $c$ is constant; otherwise we can assume $\hat c$ restricts as the identity on the set $\{0,1, \infty\}$, and that $\hat c$ has no other singularities. 

Writing $\hat c(z)=p(z)/q(z)$, it follows that $p(z)$ is divisible by $z$, there are no other zeros of $p$, that $q$ has no zeroes, and that $\hat c(1) = 1$. Thus $\hat c(z) = z = sr(1,2,3)(0,1,z)$, and as was discussed above, this implies that $h = sr(1,2,3)$ over the entire domain.

\para{\boldmath Base case 2: $n=4$}
By \Cref{factor34}, $h_*: PB_4 \to F_2$ either factors through a forgetful map $PB_4 \to PB_3$ or else through $(RQ_4)_*$. In the former case, we apply \Cref{forgetrigid} to reduce to the case $n = 3$, and so we assume that $h_*$ factors through $(RQ_4)_*$. For $\tau \in \C-\{0,1\}$ fixed, define the subset $X_\tau \subset \PSL(2,\C)$ via
 \[
X_\tau:=\{A\in \PSL(2,\C)|A(0)\neq \infty,A(1)\neq \infty,A(\infty)\neq \infty, A(\tau)\neq \infty \}.
\]
By construction, $X_\tau((0,1,\infty,\tau)) = cr(1,2,3,4)^{-1}(\tau)$. Like in the case $n = 3$, we define the holomorphic map $h_\tau: X_\tau \to \C-\{0,1\}$ by
\[
h_\tau(A) = h(A(0,1,\infty, \tau)).
\]
We claim that $h_\tau$ extends to a holomorphic map $h_\tau: \PSL(2,\C) \to \C=\{0,1\}$. To see this, we first claim that $h_\tau$ lifts to holomorphic map $H_\tau: X_\tau \to \D$. This follows from the assumption that $h_*$ factors through $(RQ_4)_* = (0,1, cr(1,2,3,4))_*$ and the characterization $X_\tau((0,1,\infty,\tau)) = cr(1,2,3,4)^{-1}(\tau)$, which shows that $h_\tau$ is homotopic to a constant map. Thus $H_\tau$ is a bounded holomorphic map. The space $X_\tau$ is the complement of hypersurfaces in the smooth complex variety $\PSL(2,\C)$. By the higher-dimensional removable singularity theorem \cite[Theorem 4.7.2]{scr}, it follows that $H_\tau$, and hence $h_\tau$, can be extended to $\PSL(2,\C)$. 

By \Cref{affinetoD}, each map $h_\tau: \PSL(2,\C) \to \C-\{0,1\}$ is a constant map $h_\tau = c(\tau)$. We can now apply the argument of the last two paragraphs in the case $n = 3$ to conclude that $c: \C-\{0,1\} \to \C-\{0,1\}$ must be an automorphism, and hence $h = cr(1,2,3,4)$ up to an automorphism of $\C-\{0,1\}$ as claimed. 

\para{\boldmath Inductive step: $n \ge 5$} We proceed by induction on $n$, taking $n = 4$ as the base case. The inductive step follows by \Cref{factor34} and \Cref{forgetrigid}.
\end{proof}

The following lemma gives a useful normalization of a map between configuration spaces. 
\begin{lem}[Normalization]
\label{normalizing}
Every holomorphic map
\[
h: \PConf_n(\C)\to \PConf_m(\C)
\]
is equivalent up to affine twisting to a unique holomorphic map 
\[ h_s:\PConf_n(\C)\to \PConf_m(\C)\]
such that the first two coordinates of $h_s(x_1,...,x_n)$ are $0,1$. 
\end{lem}
\begin{proof}
Let $p_{12}: \PConf_m(\C)\to \PConf_2(\C)$ be the projection onto the first two coordinates. Then $p_{12}\circ h$ is a holomorphic map. Let $A: \PConf_n(\C)\to \Aff$ be characterized by the condition that $A(x_1,...,x_n)(p_{12}(x_1,...,x_n)) = (0,1)$. Via the identification of complex manifolds $\PConf_2(\C) \cong \Aff$, it follows that $A$ is holomorphic. Defining $h_s$ as the affine twist $h^A$, the claim follows. 
\end{proof}

We now discuss the common values of $cr(i,j,k,l)$ and $sr(i,j,k)$.
\begin{lem}\label{zeros}\ 
\begin{enumerate}
\item 
The function $sr(1,2,3)-sr(i,j,k)$ has no zero in $\PConf_n(\C)$ if and only if $(i,j,k)$ is one of $(1,2,p), (1,p,3),(p,2,3)$ for $p \ge 4$.
\item
The function $cr(1,2,3,4)-sr(i,j,k)$ always has a zero in $\PConf_n(\C)$.
\item
The function $cr(1,2,3,4)-cr(i,j,k,l)$ has no zero in $\PConf_n(\C)$ if and only if $(i,j,k,l)$ is one of $(p,2,3,4), (1,p,3,4),(1,2,p,4),(1,2,3,p)$ for $p\neq 1,2,3,4$.
\end{enumerate}
\end{lem}
\begin{proof}
We will prove (3), the proofs of (1),(2) being similar. The function $cr(1,2,3,4)-cr(i,j,k,l)$ has a zero if and only if the expression
\begin{equation}\label{cr=cr}
\frac{x_4-x_1}{x_4-x_3} \frac{x_2-x_3}{x_2-x_1}=\frac{x_l-x_i}{x_l-x_k} \frac{x_j-x_k}{x_j-x_i}.
\end{equation}
has a solution. If $\{i,j,k,l\}=\{1,2,3,4\}$, we can see by \Cref{lem:permute} that $cr(i,j,k,l)$ is the image of $cr(1,2,3,4)$ under some element of the dihedral group $D_3$ acting on $\C-\{0,1\}$ via $z \mapsto 1-z, z \mapsto 1/z$, and it is straightforward to verify that no element of this group acts freely, implying that the equation $cr(1,2,3,4) = cr(i,j,k,l)$ admits a solution in $\PConf_n(\C)$ in this case. 

We therefore assume that index $i$ is not in $\{1,2,3,4\}$; without loss of generality, set $i=5$. If we view the equation as a function of $x_5$ (fixing other points), it will have a unique solution $z_5 = z_5(x_1, \dots, \widehat{x_5}, \dots, x_n)$ in $\CP^1$. If this is not a valid solution in $\PConf_n(\C)$, then $z_5$ must either be one of the $x_q$ for $q \ne 5$ or else $\infty$. As $z_5$ varies continuously with the parameters $x_i, i \ne 5$, if \eqref{cr=cr} has no solutions, then there must be an {\em identity} $z_5 = x_q$ for $q \ne 5$ or else $z_5 = \infty$. 

In case $z_5 = \infty$, \eqref{cr=cr} simplifies to the ostensible identity in $\C(x_1, \dots, x_n)$
\[
\frac{x_4-x_1}{x_4-x_3} \frac{x_2-x_3}{x_2-x_1}=\frac{x_j-x_k}{x_l-x_k},
\]
which is readily seen to not hold regardless of the values of $j,k,l$. If $z_5= x_q$ for some $q \ne 5$, a similar analysis shows that the only way \eqref{cr=cr} can be satisfied is if $(p,j,k,l) = (1,2,3,4)$.
\end{proof}

We now discuss similar results for $\PConf_n(\CP^1)$.

\begin{thm}\label{rigiditytofreeCP}
Every non-constant holomorphic map $f:\PConf_n(\CP^1)\to \C-\{0,1\}$ is of the form $cr(i,j,k,l)$.
\end{thm}
\begin{proof}
Let $f: \PConf_n(\CP^1)\to \C-\{0,1\}$ be a holomorphic map. Let $E_n: \PConf_{n}(\C)\to \PConf_n(\CP^1)$ be the natural embedding, which is a holomorphic map. By \Cref{rigiditytofree}, $f\circ E_n$ is a cross ratio map, either some $sr$ or $cr$. Since $E_n$ has dense image, the map $f$ is uniquely determined by $f\circ E_n$. 

If $f\circ E_n=sr(i,j,k)$, then letting $x_i$ approach $\infty$, the image under $f\circ E_n$ approaches $1$. Thus $sr(i,j,k)$ cannot extend to $\PConf_n(\CP^1)$ and so $f \circ E_n$ must be a four-term cross ratio map $cr(i,j,k,l)$, each of which extends to $\PConf_n(\CP^1)$ via the same formula. 
\end{proof}
Similarly, we have the following counterpart to \Cref{zeros}.
\begin{thm}
The equation $cr(1,2;3,4)=cr(i,j,k,l)$ has no zero in $\PConf_n(\CP^1)$ if and only if $(i,j,k,l)$ is one of $(p,2,3,4), (1,p,3,4),(1,2,p,4),(1,2,3,p)$ for $p\neq 1,2,3,4$.
\end{thm}

\subsection{\boldmath Rigidity of $\PConf_n(\C)\to \PConf_m(\C)$}
%
\begin{proof}[Proof of \Cref{Ccase}]
Let
\[
h:\PConf_n(\C)\to \PConf_m(\C)
\]
be a holomorphic map. We can assume the first two coordinates of $h(x_1,...,x_n)$ are $(0,1)$ by \Cref{normalizing}. The remaining coordinates of $h(x_1,...,x_n)$ are functions 
\[
f_3,...,f_m:\PConf_n(\C)\to \C-\{0,1\}.\]
Since the image lies in $\PConf_m(\C)$, the expressions $f_i-f_j=0$ have no solutions in $\PConf_n(\C)$.

\para{Case 1: $f_3=sr(i,j,k)$} By \Cref{zeros}, if $f_3=sr(i,j,k)$, then all other $f_i$'s are also simple ratios $sr$. Up to a permutation of coordinates on the domain, we assume that $f_3=sr(1,2,3)$. By \Cref{zeros}, $f_4$ is either $sr(1,2,4)$, $sr(1,4,3)$ or $sr(4,2,3)$. Applying \Cref{lem:permute}, by applying an affine twist and/or $NI_n$, we can assume $f_4 = sr(1,2,4)$. We then claim that $f_k=sr(1,2,k)$ for all $k \ge 3$, up to permutation on the domain. This follows from \Cref{zeros}: the only tuple $(i,j,k)$ that differs from both $(1,2,3)$ and $(1,2,4)$ in a single entry is $(1,2,p)$ for some other $p$. Then we can apply a permutation such that $p=k$. Applying the affine twist $(x_2-x_1)z + x_1$ then shows that $h$ is affine-equivalent to the forgetful map $\pi_n^m$.

\para{Case 2: $f_3=cr(i,j,k,l)$} By \Cref{zeros}, if $f_3=cr(i,j,k,l)$, then all other $f_i$'s are also four-term cross ratio functions $cr$. Applying a permutation on the domain, we assume that $f_3=cr(1,2,3,4)$. By \Cref{zeros}, $f_4$ is a $cr$ function indexed by a tuple where exactly one entry of $(1,2,3,4)$ has been replaced by $5$. As in Case 1, we apply \Cref{lem:permute} so that possibly after an affine twist and/or an application of $NI_n$, we can assume that $f_4 = cr(1,2,3,5)$. Arguing as in Case 1, it follows that $f_k = cr(1,2,3,k+1)$ for $3 \le k \le m$ up to permutation on the domain, and visibly then $h = \pi_n^m \circ RQ_n$.
\end{proof}

\subsection{Rigidity of $\PConf_n(\CP^1)\to \PConf_m(\CP^1)$ for $m\ge 3$}
\begin{proof}
Let 
\[
h: \PConf_n(\CP^1)\to \PConf_m(\CP^1)
\]
be a holomorphic map. Similar to \Cref{normalizing}, we can assume that the first three coordinates of all images of $h$ are $0, 1, \infty$ after applying a twist $A: \PConf_n(\CP^1) \to \PSL(2, \C)$. All the holomorphic maps $\PConf_n(\CP^1)\to \C-\{0,1\}$ are given by cross ratio functions $cr(i,j,k,l)$. By a similar argument as in \Cref{rigiditytofree}, it can be shown that $h$ is a forgetful map (in this setting, the maps $RQ, NI$ arise as twists and do not need to be considered separately). 
\end{proof}

\section{Pure configuration spaces in genus one}
Here we consider the setting of holomorphic maps between pure configuration spaces of Riemann surfaces where the target $Y$ has genus one. In outline, the proof is the same as that of the previous section - we construct a normalization $f_s$ of $f: \PConf_n(X) \to \PConf_m(Y)$, relative to which the component functions are tightly constrained enough to be completely classifiable. In the previous section this classification of component functions (\Cref{rigiditytofree}) relied on the group-theoretic classification \Cref{factor34}, which was proved in \cite{ChenPure}. The results of \cite{ChenPure} also treat the case where $g(X) \ge 2$ (this is recalled below in \Cref{ChenPure2}), but we must establish the case $g(X) = 1$ ourselves. It is interesting to note that whereas all three results have a similar flavor (asserting that maps from configuration spaces to hyperbolic groups factor through forgetful maps), the precise nature of these forgetful maps reflect the different genus regimes (for $g(X) = 0$ we forget all but three or four, for $g(X)=1$ we forget all but two, and for $g(X) \ge 2$ we forget all but one).

\begin{thm}\label{torus}
Any homomorphism $PB_n(T^2)\to F_m$ either factors through $PB_2(T^2)$ or has cyclic image.
\end{thm}

Using this, we will establish the main results of the section:

\begin{thm}\label{thmtorus}
Let $X,Y$ be compact Riemann surfaces with $g(X) = g(Y) = 1$. Then up to permutation of coordinates and twisting, any holomorphic map $h: \PConf_n(X)\to \PConf_m(Y)$ is induced by an isomorphism $X\to Y$ and a forgetful map.
\end{thm}

\begin{thm}\label{thmtorus2}
Let $X,Y$ be compact Riemann surfaces with $g(X) \ge 2$ and $g(Y) = 1$. Then up to twisting, any holomorphic map $h: \PConf_n(X)\to \PConf_m(Y)$ is constant.
\end{thm}

We observe that a twist $A: \PConf_n(X) \to \Aut(Y)$ is essentially the same thing as the case $m = 1$ (see \Cref{auty} below), and so \Cref{thmtorus2} really gives a reduction to the case $m = 1$. It remains to give a classification of holomorphic $h: \PConf_n(X) \to Y$. Certainly one possibility is to factor through a forgetful map $\PConf_n(X) \to X$, but there are more complicated examples, as well. For instance, if $X$ admits $Y$ as one of its isogeny factors (i.e. the Jacobian $\Jac(X)$ admits a finite cover isomorphic to a product $Y \times A$), then it is possible to use the Abel-Jacobi map to induce a map $\PConf_n(X) \to Y$ this way. We leave the problem of classifying $m = 1$ for future work.

\subsection{\Cref{torus}: from torus braid groups to free groups}
We first discuss some facts about the group $PB_n(T^2)$. Given a disk $D^2$ embedded in $T^2$, we obtain an embedding of the pure braid group $i: PB_n<PB_n(T^2)$ as a subgroup. 

We now introduce a generating set for $PB_n$. Recall that $PB_n$ is the pure mapping class group of the disk with $n$ marked points; i.e., $\pi_0(\Diff(\mathbb{D}_n))$, where $\Diff(\mathbb{D}_n)$ is the group of diffeomorphisms of $\mathbb{D}$ fixing $n$ marked points pointwise. Consider the disk with $n$ marked points $\mathbb{D}_n$ in Figure \ref{p1}. 

\begin{figure}[H]
\minipage{0.30\textwidth}
  \includegraphics[width=\linewidth]{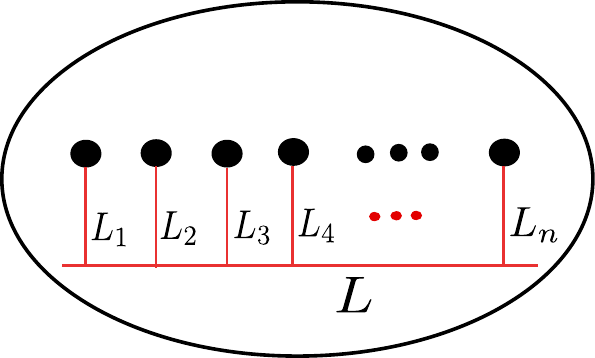}
  \caption{$\mathbb{D}_{n}$.}\label{p1}
\endminipage\hfill
\minipage{0.30\textwidth}
  \includegraphics[width=\linewidth]{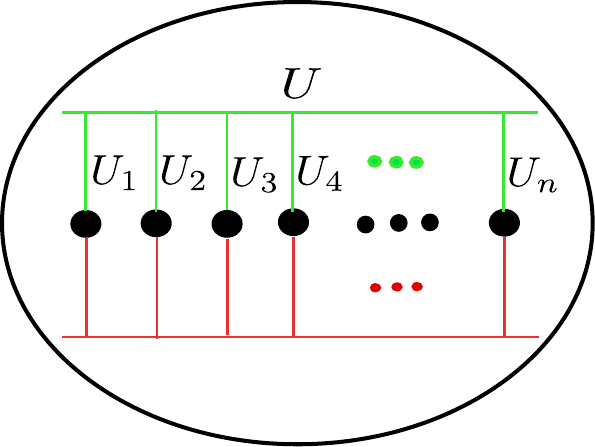}
  \caption{}\label{p3}
\endminipage\hfill
\minipage{0.30\textwidth}
  \includegraphics[width=\linewidth]{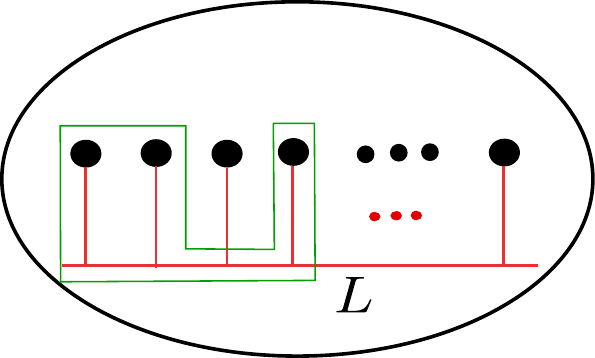}
  \caption{$a_{124}$.}
  \label{p2}
\endminipage\hfill
\end{figure}

Let $L$ be a line segment below all the marked points $x_1,...,x_n$. Let $L_1, ..., L_n$ be line segments connecting $x_1, ..., x_n$ to $L$ as in Figure \ref{p1}. Similarly, let $U$ be a line segment above all marked points and let $U_1,...,U_n$ be line segments connecting $x_1, ..., x_n$ to $U$ as shown in Figure \ref{p3}.

For $I\subset \{1, ..., n\}$, let $a_I$ (resp. $a_I'$) be the boundary curve of the tubular neighborhood of 
$\bigcup_{i \in I} L_{i}\cup L$ (resp. $\bigcup_{i \in I} U_{i}\cup U$). Let $T_{I}$ (resp. $T_{I}'$) be the Dehn twist about $a_{I}$ (resp. $a_{I}'$). Figure \ref{p2} gives  an example of a curve representing $a_{124}$. The following proposition about generating sets of $PB_n$ is classical and can be found in \cite[Theorem 2.3]{MM}.
\begin{prop}
Both $\{T_{ij} |1\le i<j\le n\}$ and $\{T_{ij}' |1\le i<j\le n\}$ are generating sets for $PB_n$.
\end{prop}

An embedding $\D_n \hookrightarrow T^2$ induces an injection $PB_n \hookrightarrow PB_n(T^2)$. We fix one such embedding, and use this to identify $T_I$ and $T_I'$ with elements of $PB_n(T^2)$. 

\begin{lem}\label{pbconj}
For $n \ge 2$, any Dehn twist $T_c$ about a simple closed curve $c$ surrounding $p_i,p_j$ is conjugate in $PB_n(T^2)$ to $T_{ij}$.
\end{lem}
\begin{proof}
The Dehn twists $T_c$ and $T_{ij}$ can be viewed as point-pushing maps about based loops $c^+$ and $c_{ij}^+$ starting from $p_i$ around $p_j$ following $c$ and $c_{ij}$ respectively. The point-pushing subgroup based at $p_i$ inside $PB_n(T^2)$ is a free group $\pi_1(T^2-\{x_1,...,\hat{x_i},...,x_n\})$. The loops $c^+$ and $c_{ij}^+$ are conjugate in $\pi_1(T^2-\{x_1,...,\hat{x_i},...,x_n\})$, since as unbased loops they both encircle the puncture $x_j$.
\end{proof}
We now prove the following statement; in anticipation of later use, we formulate it for a general compact surface, not just tori (the Dehn twists $T_{ij}$ retain the meaning given above, as loops of the $i^{th}$ point around the $j^{th}$ inside some topological embedding $\D_n \hookrightarrow X$). 
\begin{lem}\label{lemma:kernelinduced}
Let $S$ be a compact surface and let $e_n: \PConf_n(S)\to \PConf_{n-1}(S)\times S$ be the natural embedding. The kernel of the induced map $e_{n*}$ on the fundamental groups is normally generated by $\{T_{1n},...,T_{n-1,n}\}$.
\end{lem}
\begin{proof}
Endow $S$ with a complete Riemannian metric. Define $W := \PConf_{n-1}(S)\times S-\PConf_n(S)$, noting that $W$ is a union $W= \coprod W_i$ of $n-1$ disjoint embedded copies of $\PConf_{n-1}(S)$, according to the unique $i \in \{1,\dots, n-1\}$ such that $x_n = x_i$. Let $N(W)$ be a tubular neighborhood; likewise there is a decomposition $N(W) = \coprod N(W)_i$. Let $p: N(W) \to W$ be the projection taking $(x_1, \dots, x_n) \in N(W)_i$ to $(x_1, \dots, x_{n-1}, x_i)$. The map $p$ is a homotopy equivalence because it extends to a deformation retraction by radially contracting $x_n$ in to the associated $x_i$, and the set $N(W)_i^\circ := N(W)_i - W_i$ has the homotopy type of a $S^1$-bundle over $W_i$, with bundle map given by $p$. 

The use of van Kampen to obtain the result is complicated by the fact that $W$ and $N(W)$ are disconnected. Accordingly, define $Y_0 := \PConf_n(S)$ and for $i \ge 1$, 
\[
Y_i = Y_{i-1} \bigcup_{N(W)_i^\circ} N(W)_i;
\]
let $q_i \in N(W)_i^\circ$ be a basepoint.

Applying the van Kampen theorem to the above decomposition of $Y_i$ yields the following pushout diagram:
\[
\xymatrix{
\pi_1(N(W)_i^\circ,q_i)\ar[r]^{f_{i,*}}\ar[d]^{p_{i,*}} &\pi_1(Y_{i-1}, q_i)\ar[d]^{e_{i,*}}\\
 \pi_1(N(W)_i,q_i)\cong\pi_1(W_i,q_i)\ar[r] & \pi_1(Y_i,q_i)
}
\]
Therefore the kernel of $e_{i,*}$ is normally generated by the image of $\ker(p_{i,*})$ under $f_{i,*}$. Observe that $\ker(p_{i,*})$ is normally generated by the $S^1$-fiber, which corresponds under $f_{i,*}$ to the Dehn twist $T_{i,n}$. Inductively, we see that the kernel of the inclusion map $\pi_1(\PConf_n(S), q_i) \to \pi_1(Y_i,q_i)$ is normally generated by the Dehn twists $T_{j,n}$ for $j \le i$; as $Y_{n-1} = \PConf_{n-1}(S) \times S$, the result follows. 
\end{proof}

Let $P_i:=\pi_1(T^2-\{x_1,...,\hat{x_i},...,x_n\})<PB_n(T^2)$ be the point-pushing subgroup based at the point $x_i$.
\begin{lem}\label{pointpush}
The group $PB_n(T^2)$ is generated by elements in $P_i$ for $i\in \{1,...,n\}$
\end{lem}
\begin{proof}
The subgroup $P_n$ is the kernel of the homomorphism $PB_n(T^2)\to PB_{n-1}(T^2)$ induced by forgetting the last point. The lemma is deduced by induction.
\end{proof}

We will make use of the result \cite[Theorem 2.5]{ChenPure}:

\begin{thm}\label{producthyperbolic}
Let $G_1,...,G_n$ be groups and let $\Gamma<G_1\times...\times G_n$ be a finite index subgroup. Let $\pi_i:\Gamma\to G_i$ be the $i^{th}$ projection map and let  $\Gamma_i$ be the image of $\pi_i$.  Let $\Lambda$ be a torsion-free, non-elementary hyperbolic group. Then any homomorphism $\phi:\Gamma\to \Lambda$ either factors through $\pi_i$ or its image is a cyclic group. 
\end{thm}

We now start the proof of \Cref{torus}.
\begin{proof}[Proof of \Cref{torus}]
Let $\rho: PB_n(T^2)\to F_m$ be a homomorphism.

\para{\boldmath Case 1: $n=3$}
Suppose that one of $\rho(T_{12})$, $\rho(T_{13})$ or $\rho(T_{23})$ is not trivial (we will consider the situation where all three vanish below). We assume without loss of generality that $\rho(T_{12})\neq 1$. Let us consider the centralizer of $T_{12}$. The point push of the third point gives the embedding $P_3: \pi_1(T^2-\{x_1,x_2\})\to PB_3(T^2)$. Denote the based loop at $x_3$ corresponding to $T_{12}$ as $c$. Then $P_3(c) = T_{12}^{-1}T_{123}$. The loop $c = [a,b]$ is a commutator in $\pi_1(T^2-\{x_1,x_2\})$, where $a,b$ are standard generators for $\pi_1(T^2, x_3)$ disjoint from $T_{12}$, and so crucially, both $P_3(a)$ and $P_3(b)$ {\em commute} with $T_{12}$. As the centralizer of any nontrivial element of $F_m$ is cyclic, it follows that $\rho([P_3(a), P_3(b)]) = \rho(T_{12}^{-1}T_{123})=1$ and hence $\rho(T_{123})=\rho(T_{12})$. By the same logic, either $\rho(T_{23}) = 1$ or $\rho(T_{23}) = \rho(T_{123})$, and likewise either $\rho(T_{13}) = 1$ or else $\rho(T_{13}) = \rho(T_{123})$.

By the lantern relation, $T_{12}T_{23}T_{13} = T_{123}$. If either $\rho(T_{23})$ or $\rho(T_{13})$ is trivial, this implies that {\em both} are. The other possibility is that all three of $\rho(T_{12}), \rho(T_{23}), \rho(T_{13})$ equal $\rho(T_{123})$, which implies that $\rho(T_{123})^2 = 1$; as $F_m$ is torsion-free, this implies $\rho(T_{123}) = 1$, contrary to the assumption that $\rho(T_{12}) \ne 1$. We conclude that $\rho(T_{13}) = \rho(T_{23}) = 1$, so by \Cref{lemma:kernelinduced}, it follows that $\rho$ factors through the product $PB_2(T^2) \times T^2$. By \Cref{producthyperbolic}, we conclude that $\rho$ either factors through $PB_2(T^2)$ or else has cyclic image. Note that the argument of this paragraph covers the case where all of $\rho(T_{12})$, $\rho(T_{13})$, $\rho(T_{23})$ are trivial.

\para{\boldmath Case 2: $n=4$}
As above, we can assume without loss of generality that $\rho(T_{12})\neq 1$. By \Cref{factor34}, the restriction $\rho|_{PB_4}$ either factors through a forgetful map to $PB_3$ or $RQ_4$. In the first case, $\rho(T_{14})=\rho(T_{24})=\rho(T_{34})= 1$, which implies that $\rho$ factors through $e_{4*}$ by \Cref{lemma:kernelinduced}. Applying \Cref{producthyperbolic}, either the image is cyclic or else we have reduced to the case $n =3$.

Suppose then that $\rho$ factors through $RQ_4$. On the level of $B_4$, the map $RQ_4$ sends the standard generators $\sigma_1, \sigma_2$ to themselves and $\sigma_3$ to $\sigma_1$. Thus $\rho(T_{34})= \rho(\sigma_3^2) = \rho(\sigma_1^2) = \rho(T_{12})$.

\para{\boldmath Subcase 1: $\rho(T_{23})$ does not commute with $\rho(T_{12})$}
We first claim that under this assumption,
\[\rho(T_{123})=\rho(T_{1234})=\rho(T_{234})=1.\]
To see this, observe that each such element commutes with both $\rho(T_{12}) = \rho(T_{34})$ and $\rho(T_{23})$ and hence lies in the intersection of their centralizers, which is trivial.

\begin{center}
\begin{figure}[H]
\labellist
\small
\pinlabel $\gamma$  at 90 91
\pinlabel $\gamma_3$ at 75 90
\pinlabel $\gamma_4$ at 60 90
\pinlabel $\gamma_{34}$ at 40 90
\pinlabel $\bar{c}$ at 170 30
\endlabellist
\includegraphics{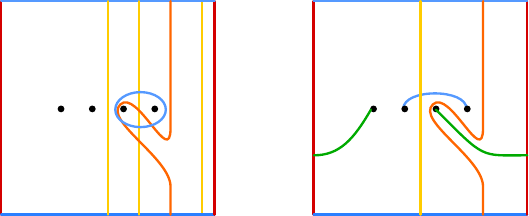}
\caption{The configuration of curves used in Subcase 1.}
\label{figure:ccurve}
\end{figure}
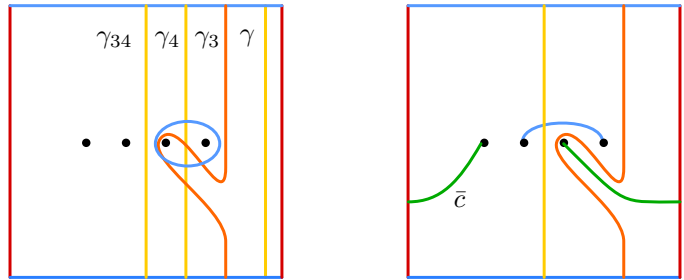
\end{center}

Referring to the left side of \Cref{figure:ccurve}, we define the elements $S, S_3, S_4, S_{34} \in \Mod(T^2, \{x_1, \dots, x_4\})$ as the Dehn twists about the curve $\gamma$ with the corresponding subscript. Then there is a lantern relation of the form
\[
T_{34} S_4 S_3 = SS_{34}
\]
We rearrange, expressing each side as an element of $PB_4(T^2)$:
\[
T_{34} S_4 S^{-1} = S_3^{-1} S_{34}
\]
Since $S_4 S^{-1}$ commutes with $T_{12}$ and $T_{23}$, it follows that $\rho(S_4 S^{-1})=1$ since $\rho(T_{12})$ and $\rho(T_{23})$ do not commute by hypothesis. Thus $\rho(S_3^{-1} S_{34}) = \rho(T_{34})$. There is a second lantern relation $T'_{24}T_{34}T_{23}=T_{234}$, from which we conclude $\rho(T'_{24}) = \rho(T_{34}T_{23})^{-1}$.

Let $c$ be the curve given as a regular neighborhood of the arc $\bar c$ indicated on the right side of \Cref{figure:ccurve}. By disjointness, $T_c$ commutes with $T'_{24}$ and $S_3^{-1} S_{34}$. However $\rho(S_3^{-1} S_{34}) = \rho(T_{34}) = \rho(T_{12})$ and $\rho(T'_{24}) = \rho(T_{34}T_{23})^{-1} = \rho(T_{12}T_{23}^{-1})$ don't commute, which implies that $\rho(T_c)=1$. By \Cref{pbconj}, $\rho(T_c)$ is conjugate to $\rho(T_{13})$, implying that $\rho(T_{13}) = 1$ as well. However, there is a lantern relation $T_{12}T_{23}T_{13} = T_{123}$, and since $\rho(T_{123}) = 1$, this implies that $ \rho(T_{13}) = \rho(T_{12}T_{23})^{-1} \ne 1$, a contradiction. 

\begin{center}
\begin{figure}[H]
\includegraphics{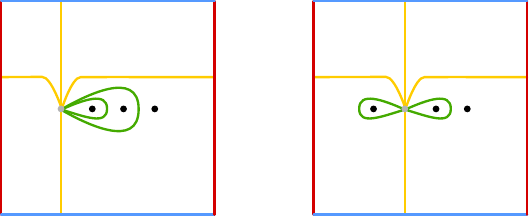}
\caption{Subcase 2: generators for $P_1$ and $P_2$.}
\label{figure:torusgens}
\end{figure}
\end{center}
\para{\boldmath Subcase 2: $\rho(T_{23})$ commutes with $\rho(T_{12})$} In this case, $\rho(PB_4)$ is in the centralizer of $\rho(T_{12})$. We now show that $\rho(PB_4(T^2))$ is in also in the centralizer of $\rho(T_{12})$, and hence $\rho$ has cyclic image. By \Cref{pointpush}, it suffices to show that $\rho(P_i)$ is in the centralizer of $\rho(T_{12})$ for $i = 1, \dots, 4$. The subgroup $P_1$ is the fundamental group of $T^2-\{x_2,x_3,x_4\}$ and is generated by point-push maps about the paths indicated in \Cref{figure:torusgens}. Expressing a point push as the product of Dehn twists about the boundary components of a regular neighborhood, one observes that each such curve is disjoint from either $a_{12}$ or $a_{34}$ with the exception of $T_{23}$. By hypothesis, $\rho(T_{23})$ commutes with $\rho(T_{12})$, and it follows that $\rho(P_1)$ commutes with $\rho(T_{12}) = \rho(T_{34})$ as claimed. A similar analysis of the generating sets for $P_2, P_3, P_4$ show that the same result holds, completing the argument. 

\para{\boldmath Case 3: general $n$} This proceeds as in the first step of the case $n = 4$. If all $\rho(T_{ij}) = 1$, then the image is abelian and hence cyclic. Otherwise, $\rho(T_{12}) \ne 1$ without loss of generality. By \Cref{factor34}, the restriction $\rho|_{PB_n}$ factors through a forgetful map, so that (without loss of generality) $\rho(T_{i,n}) = 1$ for $1 \le i \le n-1$. By \Cref{lemma:kernelinduced}, it follows that $\rho$ factors through $e_{n,*}$, and applying \Cref{producthyperbolic} implies that either $\rho$ has cyclic image or else we have reduced to the case $m = n-1$.
\end{proof}

\subsection{\Cref{thmtorus}: holomorphic maps between configuration spaces on elliptic curves}
To begin, we recall the classification of maps $f: PB_n(X) \to F_m$ established (in greater generality) in \cite[Theorem 1.1]{ChenPure}.

\begin{thm}\label{ChenPure2}
Let $X$ be a compact Riemann surface with $g(X) \ge 2$, and let $F_m$ be a free group of rank $m \ge 2$. Then every homomorphism $f: PB_n(X) \to F_m$ either has cyclic image or else factors through a forgetful map $p_i: PB_n(X) \to X$. 
\end{thm}

Using this and \Cref{torus}, we establish the following analogue of \Cref{rigiditytofree}.
\begin{lem}\label{lemtorus}
Let $X, Y$ be compact Riemann surfaces, with $g(Y) = 1$. Choose a base point $O \in Y$, thereby endowing $Y$ with a group structure. Let $h: \PConf_n(X)\to Y-\{O\}$ be a holomorphic map. Then either $h$ is constant or there is an isomorphism $I:X\cong Y$, and $h=I(x_i-x_j)$ for some $i,j$.
\end{lem}
\begin{proof}
We proceed by induction, the case $n = 1$ being trivial (as every holomorphic map $h: X \to Y - \{O\}$ is constant).  Now let $n \ge 2$ be given. By \Cref{torus} and \Cref{ChenPure2}, either $h_*$ has cyclic image or $h_*$ factors through a forgetful map: either $PB_n(X)\to PB_2(X)$ in the case of $g(X) = 1$ or $PB_n(X) \to X$ in the case of $g(X) \ge 2$. 

Firstly, suppose that $h_*$ has cyclic image. Fix a configuration $C=(x_1,...,x_{n-1})\in \PConf_{n-1}(X)$, and for $1 \le k \le n$, consider the natural embedding $i_{C,k}: X-C\subset \PConf_n(X)$ where all but the $k^{th}$ coordinate is fixed at $C$; note this is a holomorphic map. By \Cref{extend}, $h \circ i_{C,k}$ extends to a holomorphic map $H: X\to Y$. If $H$ is not constant, then Hodge theory asserts that $H^1(Y;\C)$ is spanned by holomorphic differentials and their conjugates. Since nonzero holomorphic differentials pull back to nonzero holomorphic differentials along holomorphic maps, it follows that $H^*:H^1(Y;\C)\to H^1(X;\C)$ is injective, and dually $H_*: H_1(X;\C)\to H_1(Y;\C)$ is surjective. This contradicts the assumption that $h \circ i_{C,k}$ has cyclic image.

Next, suppose that that $g(X) \ge 2$ and $h_*$ factors through a forgetful map $p: PB_n(X) \to X$; without loss of generality, assume that $p$ is the projection onto the $n^{th}$ factor. As in the previous paragraph, we fix a configuration $C=(x_1,...,x_{n-1})\in \PConf_{n-1}(X)$ and consider $h \circ i_{C,k}: X-C \to Y-\{O\}$. By hypothesis, for $1 \le k \le n-1$, the induced map $(h \circ i_{C,k})_*$ on fundamental group is trivial and so there is a lift $H: X-C \to \D$. By the removable singularity theorem, this extends to a holomorphic map $H: X \to \D$, but this must be constant by the maximum principle. We conclude that $h: \PConf_n(X) \to Y-\{O\}$ factors through $p_n: \PConf_n(X) \to X$. The induced map $h: X \to Y$ then has degree zero (since it misses $O \in Y$ by hypothesis) and hence is constant, as claimed.

Finally, suppose that $g(X) = 1$ and $h_*$ factors through a forgetful map $PB_n(X)\to PB_2(X)$; for notational simplicity, we may therefore assume that $n = 2$. Let $x_1\in X$ and $i_{x_1}: X-\{x_1\}\to \PConf_2(X)$ be the natural embedding. Then the induced map $h\circ i_{x_1}$ is also holomorphic. By the removable singularity theorem (\Cref{extend}), $h\circ i_{x_1}$ can be uniquely extended to a map $\overline{h\circ i_{x_1}}: X\to Y$. Every holomorphic map between elliptic curves is either constant or else a covering map. Since $O$ has at most a single preimage, $\overline{h\circ i_{x_1}}$ is either an isomorphism or a constant map. 

If $\overline{h\circ i_{x_1}}$ is a constant map, then this holds for all $x_1\in X$. Thus $h$ factors through the forgetful map to the first coordinate $X$, reducing further to the previous case $n = 1$, from which we conclude that $h$ is constant.

If  $\overline{h\circ i_{x_1}}$ is an isomorphism, then we obtain a family of isomorphisms $\overline{h\circ i_{x_1}}:X\to Y$ such that $\overline{h\circ i_{x_1}}(x_1)=O$. Thus we obtain a global holomorphic map $H:X\times X\to Y$ such that $H(x_1,x_1)=O$. Choose a basepoint $O_X \in X$, and let $I=\overline{h\circ i_{O_X}}$. Thus $H(O_X,x_2)=I(x_2)$. A holomorphic map $X\to Y$ is uniquely determined by the map on the fundamental group and the image of a single point. Thus $H(x_1,x_2)=I(x_2-x_1)$ for any other $x_1$ since this map satisfies that $H(x_1,x_1)=O$ and is compatible on fundamental groups. 
\end{proof}

\begin{proof}[Proof of \Cref{thmtorus}]
Let $h: \PConf_n(X) \to \PConf_m(Y)$ be a holomorphic map, and fix a basepoint $O\in Y$. Let $y_1 \in Y$ denote the first coordinate of $h(x_1, \dots, x_n)$; note that $y_1$ is a holomorphic function of $x_1, \dots, x_n$. Exploiting the group structure, define the normalization 
\[
h_s(x_1,...,x_n)=h(x_1,...,x_n)-(y_1,...,y_1).
\] Thus, $h_s(x_1,...,x_n)=(O,f_2(x_1,...,x_n),...,f_m(x_1,...,x_n))$ where $f_i: \PConf_n(X)\to Y-\{O\}$ are holomorphic maps. By \Cref{lemtorus}, it follows that $f_p(x_1,...,x_n)=I_p(x_i-x_j)$ for some isomorphism $I_p:X\to Y$ and $1\le i<j\le n$. 

To proceed, we consider the set of all isomorphisms $I: X \to Y$. This set is a torsor for $\Aut(Y)$, the group of automorphisms of $Y$ as a Riemann surface. There is a semi-direct product structure
\begin{equation}\label{auty}
\Aut(Y) \cong (Y, O) \rtimes \Aut(Y, O),
\end{equation}
where $(Y,O)$ indicates the elliptic curve structure on $Y$ with $O$ as the origin, and $\Aut(Y,O)$ indicates the subgroup of $\Aut(Y)$ fixing $O$, equivalently the subgroup of {\em group} automorphisms. The group $\Aut(Y,O)$ is always finite, and always contains the negation map $-id: x \mapsto -x$. 

Fix the identification $I_2: X \to Y$ associated with the first nontrivial coordinate $f_2$ once and for all; accordingly, we suppress $I_2$ from the notation. For $p \ge 3$, the torsor structure then gives expressions $I_p = (\epsilon_p, \alpha_p) \circ I_2$, where $\epsilon_p \in Y$ and $\alpha_p \in \Aut(Y, O)$. Succinctly, $I_p(x) = \alpha_p(x)+ \epsilon_p$, so that we can write (after a permutation of coordinates in the domain)
\[
f_s(x_1, \dots, x_n) = (O, x_2-x_1, \alpha_3(x_{i_3})-\alpha_3(x_{j_3}) + \epsilon_3, \dots, \alpha_m(x_{i_m}) - \alpha_m(x_{j_m}) + \epsilon_m).
\]
By hypothesis, the difference $\alpha_p(x_{i_p})-\alpha_p(x_{j_p}) + \epsilon_p - (x_2-x_1)$ does not have $O$ in its image, so that by \Cref{lemtorus}, there is an identity
\[
\alpha_p(x_{i_p})-\alpha_p(x_{j_p}) + \epsilon_p - (x_2-x_1) = \beta(x_k) - \beta(x_l) + \delta,
\]
for distinct indices $k,l$ and $(\beta, \delta) \in \Aut(Y)$, or equivalently
\[
x_1 + \alpha_p(x_{i_p}) + \beta(x_l) - (x_2 + \alpha_p(x_{j_p}) + \beta(x_k)) = \delta-\epsilon_p.
\]
The right-hand side above is constant, and so the left-hand side must exhibit cancellation. If either $\beta$ or $\alpha_p$ is not $\pm id$, then this is not possible (e.g. lifting to the universal cover $\C$ of $Y$, the derivative of the left-hand side would necessarily be everywhere nontrivial). Up to an exchange of indices, we can assume that $\alpha_p = \beta= id$, so that
\[
x_1 + x_{i_p} + x_l - (x_2 + x_{j_p} + x_k) = \delta-\epsilon_p.
\]
It follows that exactly one of the conditions $i_p = 2$ or $j_p = 1$ holds (they cannot both hold simultaneously since then $k = l$). By a permutation of coordinates, we can assume $j_3 = 1$ and $i_3 = 3$. At this point, our expression for $f_s$ has become
\[
f_s(x_1, \dots, x_n) = (O, x_2-x_1, x_3 - x_1 + \epsilon_3, \dots, x_{i_m} - x_{j_m} + \epsilon_m),
\]
subject to the condition that exactly one of the identities $i_p = 2$ or $j_p = 1$ holds for all $p$. By comparing the second and $p^{th}$ entries, we conclude that $\epsilon_p = O$ for all $p \ge 3$. 

We claim that necessarily $j_p = 1$ must hold for all $p \ge 4$. Suppose to the contrary; then (without loss of generality)
\[
f_s(x_1, \dots, x_n) = (O, x_2-x_1, x_3-x_1, x_2-x_{j_4}, \dots). 
\]
By the previous analysis, comparing the third and fourth component forces $j_4 = 1$, but then the fourth component equals the second, a contradiction. Finally, we can re-normalize $f_s$ by translation by $x_1$, showing that $f$ is given by a forgetful map as claimed. 
\end{proof}

\subsection{\Cref{thmtorus2}: from higher genus to genus one}
 \begin{proof}[Proof of \Cref{thmtorus2}]
 Write $h = (f_1, \dots, f_m)$, with each $f_k: \PConf_n(X) \to Y$ holomorphic. Choose a basepoint $O \in Y$, and, as in \Cref{thmtorus}, normalize $h: \PConf_n(X) \to \PConf_m(Y)$ by the twist 
 \[
 A(x_1,\dots, x_n)(y) = y-f_1(x_1, \dots, x_n)
 \]
 so that $h = (O, f_2, \dots, f_m)$ with $f_k: \PConf_n(X) \to Y-O$ holomorphic. By \Cref{lemtorus}, each $f_k$ must be constant.
 \end{proof}

\section{Pure configuration spaces in higher genus}

In this section we consider the problem of classifying holomorphic maps $h: \PConf_n(X) \to \PConf_m(Y)$, where $X,Y$ are Riemann surfaces of higher genus. We find that the situation is quite rigid - \Cref{mainpure} shows that either $X = Y$ and the map is forgetful, or else $h$ is induced from a family of maps $f_i: X \to Y$ which pairwise have disjoint graphs. In \Cref{mainbound}, we consider the question of how many such $f_i$ can exist, which is essentially a variant of the de Franchis theorem from classical algebraic geometry.

\subsection{Classification of holomorphic maps}

\begin{thm}\label{mainpure}
Let $X$ and $Y$ be compact Riemann surfaces, with $g(X), g(Y) \ge 2$. Suppose
\[
h: \PConf_n(X) \to \PConf_m(Y)
\]
is a nonconstant holomorphic map. Then up to permutation of coordinates and the actions of $\Aut(X)$ and $\Aut(Y)$, either $X = Y, n \ge m$, and $h$ is a forgetful map $\PConf_n(X) \to \PConf_m(X)$, or else $h$ factors as the composition of a forgetful map $\PConf_n(X) \to X$ and a holomorphic map $X \to \PConf_m(Y)$. 
\end{thm}

\begin{proof}
Consider the composition $p_i \circ h: \PConf_n(X) \to Y$, where $p_i: \PConf_m(Y) \to Y$ is the projection onto the $i^{th}$ factor. 
\begin{claim}\label{eachbig}
Each $p_i \circ h$ is nonconstant and induces a surjection $H_1(\PConf_n(X);\Q) \to H_1(Y; \Q)$.
\end{claim}
\begin{proof}
We first show that each $p_i \circ h$ is nonconstant. If $m = 1$ and $p_1$ is constant then $h$ is constant, contrary to hypothesis. If $m > 1$, then supposing that any $p_i \circ h$ is constant (say with value $y_0 \in Y$) let $j \ne i$ be some other index, and consider $p_j \circ h$. If this is nonconstant, then let $C \subset X$ be a configuration of $n-1$ distinct points, and consider the inclusion $i_{C,k}: X-C \subset \PConf_n(X)$ as before.  Here $k$ is chosen so as to make the composition $p_j \circ h \circ i_{C,k}$ nonconstant. By the removable singularity theorem (\Cref{extend}), this extends to give a nonconstant holomorphic map $f_C: X \to Y$ and by the de Franchis theorem, the space of such maps is discrete, and hence $f:=f_C$ is independent of $C$. Since $f: X \to Y$ is nonconstant, it has positive degree and hence is surjective. Thus, there is some $(x_1, \dots, x_n) \in \PConf_n(X)$ such that $p_j h(x_1, \dots, x_n) = y_0 =  p_i h(x_1, \dots, x_n)$, i.e. the map fails to have codomain $\PConf_n(Y)$, contrary to assumption. 

To see that $(p_i \circ h)_*$ is surjective, it suffices to see that $f_*: H_1(X; \Q) \to H_1(Y;\Q)$ is surjective, where $f$ is as in the above paragraph, but this is a general property of nonconstant holomorphic maps between compact Riemann surfaces, following from the existence of a transfer homomorphism $f^!: H_1(Y;\Q) \to H_1(X;\Q)$ with the property that $f^! f_* = \deg(f) I_{H_1(X;\Q)}$ (see, e.g. \cite[Definition 2.2]{tanabe} for a dual cohomological formulation).
\end{proof}

We consider the induced map $(p_i \circ h)_*: \PB_n(X) \to \pi_1(Y)$. By \cite[Lemma 2.5]{lei}, either $(p_i \circ h)_*$ factors as $(p_i \circ h)_* = f_* \circ p_{j,*}$ for some $f_*: \pi_1(X) \to \pi_1(Y)$,  where $p_{j,*}: \PB_n(X) \to \pi_1(X)$ is induced by projecting onto the $j^{th}$ factor, or else $(p_i \circ h)_*$ has cyclic image, possibly trivial. The latter case cannot happen, since such maps would not induce a surjection on $H_1$, contrary to \Cref{eachbig}. 

Defining $[k] := \{1,\dots, k\}$, we next claim that there is a function $j: [m] \to [n]$ and nonconstant holomorphic maps $\alpha_i: X \to Y$ for which the diagram below commutes for all $i$:
\begin{equation}\label{cdiagram}
\xymatrix{
\PConf_n(X) \ar[r]^{h} \ar[d]_{p_{j(i)}} & \PConf_m(Y) \ar[d]^{p_i} \\
X \ar[r]_{\alpha_i}				& Y
}
\end{equation}
The function $j$ can be defined as follows: by the above, $(p_i \circ h)_*: \PB_n(X) \to \pi_1(Y)$ factors through some projection $p_{j,*}: \PB_m(X) \to \pi_1(X)$; let $j(i)$ be this $j$. We observe that $h$ is determined by its values on the collection of submanifolds $(X-C)_k$. For $k \ne j(i)$, the restriction of the holomorphic map $p_i \circ h$ to $(X-C)_k$ is nullhomotopic by the preceding paragraph, and hence constant. Thus $p_i \circ h$ factors through $p_{j(i)}$ as required.

By construction, if $j$ is constant, then $h$ factors through some projection $p_j: \PConf_n(X) \to X$. It remains to show that if $j$ is nonconstant, then $Y = X, n \ge m$, and $h$ is a forgetful map up to permutation of coordinates and the application of some $\alpha \in \Aut(X)$ to each component of $\PConf_n(X)$. 

We claim that if $j$ is nonconstant, then $\alpha_{i_1} = \alpha_{i_2}$ for all pairs of indices. Supposing to the contrary, without loss of generality, we may take $i = 1, j = 2$, and $j(1) = 1, j(2) = 2$.  Let $x_1 \in X$ be given such that $\alpha_1(x_1) \ne \alpha_2(x_1)$, and let $x_2 \in X$ satisfy $\alpha_2(x_2) = \alpha_1(x_1)$. Completing $x_1, x_2$ to a point $(x_1, x_2, \dots, x_n) \in \PConf_n(X)$, we see that $h(x_1, \dots, x_n) = (\alpha_1(x_1), \alpha_2(x_2), \dots, \alpha_n(x_{j(n)}))$ has a repeated entry $\alpha_1(x_1) = \alpha_2(x_2)$, a contradiction. 

To see that $X = Y$ in this case, we show that the fixed map $\alpha: X \to Y$ has degree $1$. If not, let $x_1 \ne x_2 \in X$ satisfy $\alpha(x_1) = \alpha(x_2) = y$; then $h(x_1, x_2, \dots) = (y,y, \dots)$ has a repeated entry. Thus $X = Y$, and by adjusting by $\alpha^{-1}$ if necessary, we may assume that $\alpha = id$. Now it is clear that $n \ge m$, otherwise $h(x_1, \dots, x_n) = (x_{j(1)}, \dots, x_{j(n)})$ would have a repeated entry.
\end{proof}

\subsection{Bounds}
To complement \Cref{mainpure}, we consider the problem of determining the maximal $m$ for which there is a nonconstant holomorphic map $h: X \to \PConf_m(Y)$. This is closely related to the {\em effective de Franchis problem}, which asks for bounds on the number of distinct holomorphic maps $f_i: X \to Y$ (known to be finite for $g(X), g(Y) \ge 2$ by the classical de Franchis theorem); here we add the condition that the images of $f_i$ be pairwise-disjoint (pairwise {\em have no coincidences}, in the terminology of \cite{tanabe}). The general effective de Franchis problem is far from conclusively resolved - Chamizo \cite{chamizo} obtains an upper bound that is slightly larger than exponential in $g(X)$, while the largest known examples are linear in the genus (arising when $X$ and/or $Y$ has a large automorphism group which can be used to enlarge the number of morphism by pre/post composition).

Here, we find that the condition that the morphisms pairwise have no coincidences imposes a strong constraint, greatly reducing the upper bound, although in practice there is still a gap between the upper bound of \Cref{mainbound} and the largest known examples (arising when $Y$ is equipped with a group of free automorphisms). 

\begin{thm}\label{mainbound}
Let $X, Y$ be compact Riemann surfaces each of genus at least $2$, and let $h: X \to \PConf_m(Y)$ be a nonconstant holomorphic map. Then $m \le 4g(X)g(Y)$. 
\end{thm}

\begin{proof}
Each holomorphic map $f: X\to Y$ induces $f_* \in \Hom(J(X), J(Y))$, the induced map on Jacobians; Martens observes \cite{martens} that distinct morphisms $f, g$ induce distinct maps $f_*, g_*$ so long as $g(Y) \ge 2$. Tanabe \cite[Definition 2.9]{tanabe}, following ideas of Fuertes -- González-Diez \cite{FGD}, Martens \cite{martens} and ultimately Weil \cite{weil}, introduces a certain positive-definite inner product $\pair{\cdot, \cdot}$ on $\Hom(J(X), J(Y))$. According to \cite[Theorem 4.1]{tanabe}, if $f,g: X \to Y$ are holomorphic and have no {\em coincidences} (i.e. $f(x) \ne g(x)$ for all $x \in X$), then $\deg(f) = \deg(g)$ and $\cos(f_*, g_*) = g(Y)^{-1}$, where $\cos(v,w) : = \pair{v,w}/\norm{v}\norm{w}$ is defined as in any inner product space. If $h: X \to \PConf_m(Y)$ is given, the component functions $h_1, \dots, h_m$ pairwise have no coincidences, and thus determine a configuration of vectors $h_{1,*}, \dots, h_{m,*} \in \Hom(J(X), J(Y))$ where the angles $\cos(h_{i,*}, h_{j,*}) = g(Y)^{-1}$ are pairwise fixed and equal. 

As $\Hom(J(X), J(Y))$ is a subgroup of $\Hom(H_1(X;\Z), H_1(Y,\Z)) \cong \Z^{4g(X)g(Y)}$, to prove the claim, it suffices to show that such a configuration of vectors must be linearly independent. It suffices to consider the associated unit vectors $v_1, \dots, v_m$. Let $A$ be the matrix with $A_{ij} = \pair{v_i, v_j}$; linear-independence of $\{v_1, \dots, v_m\}$ is equivalent to the nonsingularity of $A$. By hypothesis, 
\[
A = (1-g(Y)^{-1})I + C,
\]
where $C$ is the matrix where every entry is given by $g(Y)^{-1}$. The eigenvalues of $C$ are $0$ and $mg(Y)^{-1}$, and hence the eigenvalues of $A$ are $1-g(Y)^{-1}$ and $1+ (m-1)g(Y)^{-1}$. As $g(Y) \ge 2$, both of these are nonzero, which proves the claim. 
\end{proof}

\section{Distinct genus regime}

In this section, we examine what happens when $g(X)$ and $g(Y)$ belong to distinct genus regimes (i.e. $g = 0, 1,$ or $g \ge 2$) - this is the setting in which there are very few holomorphic maps. The results largely follow from basic principles, but we include the proofs for the sake of completeness.

\begin{prop}\label{lowtohigh}
Let $X,Y$ be Riemann surfaces of finite type with $g(Y) > g(X)$. Then every holomorphic map $h: \PConf_n(X) \to \PConf_m(Y)$ is constant.
\end{prop}
\begin{proof}
Fix a configuration of distinct points $C = \{x_1, \dots, x_{n-1}\} \subset X$, and for $1 \le k \le n$, consider the inclusions $(X-C)_k \hookrightarrow \PConf_n(X)$ as in the proof of \Cref{mainpure}. Composing with the projection $p_i: \PConf_m(Y) \to Y$ onto the $i^{th}$ factor, we obtain a holomorphic map $f_{ik}: X-C \to Y$. As $g(Y) > g(X)$, any such map must be constant; varying $i,k$, and $C$ then shows that $h$ itself is constant. 
\end{proof}

\begin{prop}\label{clowhigh}
Let $Y$ be a compact Riemann surface of genus $g(Y) = 1$. Then every holomorphic map $h: \Conf_n(\CP^1) \to \Conf_m(Y)$ is constant.
\end{prop}

\begin{proof}
We consider the induced map on fundamental group
\[
h_*: B_n(\CP^1) \to \pi_1(Y);
\]
note that as the target is abelian, this factors through $H_1(B_n(\CP^1);\Z)$. It is well-known that $B_n(\CP^1) = B_n(S^2)$ is a quotient of $B_n$ by the word $\sigma_1 \sigma_2 \dots \sigma_{n-1} \sigma_{n-1} \dots \sigma_1$. Thus, $H_1(B_n(\CP^1); \Z) \cong \Z/(2n-2)\Z$. As $\pi_1(Y) \cong \Z^2$ is torsion-free, it follows that $h_*$ is the trivial map. Applying \Cref{lemma:kernelinduced}, we conclude that $h$ extends to a holomorphic map
\[
H: (\CP^1)^n \to Y.
\]
Every such map is constant (e.g. $H$ lifts to the universal cover $\C$ of $Y$ and is therefore constant via the maximum principle).
\end{proof}

We remark that in the setting of $g(Y) \ge 2$, the same argument (replacing $Y$ with its Jacobian) shows that any holomorphic map $h: \Conf_n(\CP^1) \to \Conf_m(Y)$ has image contained in a fixed linear system on $Y$, but this by itself is not enough to conclude that $h$ itself is constant.

\bibliographystyle{alpha}
\bibliography{citing}{}
\end{document}